\documentclass[11pt]{amsart}
\usepackage{amsmath,amssymb,amsthm,mathtools}
\usepackage[left=2.5cm,right=2.5cm,top=2.5cm,bottom=2.5cm]{geometry}
\usepackage{mathrsfs}
\usepackage{enumitem}
\usepackage{xcolor}
\usepackage{tikz}
\usepackage{hyperref}
\usepackage{cite}
\hypersetup{
 colorlinks=true,
 linkcolor=blue,
 citecolor=blue,
 filecolor=blue,
 urlcolor=blue,
 pdftitle={A new approach to compactly supported stationary Euler flows}
}
\usepackage{verbatim}

\newtheorem{theorem}{Theorem}[section]
\newtheorem{proposition}[theorem]{Proposition}

\newtheorem{lemma}[theorem]{Lemma}
\theoremstyle{definition}

\newtheorem{remark}[theorem]{Remark}
\numberwithin{equation}{section}

\renewcommand{\epsilon}{\varepsilon}
\newcommand{\RR}{\mathbb R}
\newcommand{\NN}{\mathbb N}

\newcommand{\dd}{\,\mathrm d}
\newcommand{\norm}[1]{\left\|#1\right\|}

\DeclareMathOperator{\dist}{dist}
\DeclareMathOperator{\supp}{supp}
\DeclareMathOperator{\Div}{div}
\DeclareMathOperator{\curl}{curl}
\DeclareMathOperator{\Id}{Id}

\newcommand{\loc}{\mathrm{loc}}
\newcommand{\step}[1]{\smallskip\noindent\textit{#1}\quad}
\renewcommand\leq\leqslant
\renewcommand\geq\geqslant
\setlist[enumerate]{label=\textup{(\roman*)},leftmargin=*}

\begin{document}

\title[Compactly supported stationary Euler flows]{A new approach to compactly\\ supported stationary Euler flows}

\author[A. Enciso]{Alberto Enciso}
\address{   
\newline
\textbf{{\small Alberto Enciso}} 
\vspace{0.15cm}
\newline \indent Instituto de Ciencias Matem\'aticas, Consejo Superior de Investigaciones Cient\'\i ficas, 28049 Madrid, \newline \indent Spain}
\email{aenciso@icmat.es}

 \author[A. J. Fern\'andez]{Antonio J.\ Fern\'andez}
 \address{ \vspace{-0.4cm}
\newline 
\textbf{{\small Antonio J. Fern\'andez}} 
\vspace{0.15cm}
\newline \indent Departamento de Matem\'aticas, Universidad Aut\'onoma de Madrid, 28049 Madrid, Spain}
 \email{antonioj.fernandez@uam.es}

\date{}

\begin{abstract}
\vspace{-0.2cm}
We construct smooth, compactly supported stationary Euler flows in two
and three dimensions. The three-dimensional flows are axisymmetric with
swirl, and their pressure varies along some streamlines. The planar
flows are not locally radial. Both constructions start from simple radial profiles, which we perturb by solving scalar elliptic equations on fixed domains. We choose these equations so that their solutions can be exactly constant on open sets and meet these constant regions smoothly. Comparison arguments produce the constant regions, while an a priori regularity theorem ensures that all derivatives vanish where a solution meets them. This gives compact support and smoothness without having to describe the boundary of the support. \vspace{-0.8cm}
\end{abstract}

\maketitle

\section{Introduction}
We consider the stationary incompressible Euler equations
\begin{equation}\label{eq:euler}
 (u\cdot\nabla)u+\nabla p=0,\qquad \Div u=0
 \quad\textup{in }\RR^d, \qquad d \in \{2,3\}\,. 
\end{equation}
Gavrilov~\cite{Gavrilov} established the existence
of nonzero, compactly supported smooth solutions in three dimensions
in a beautiful and surprising construction. His proof is remarkably
short, yet the identities that make it work are far from obvious.
Constantin, La and Vicol~\cite{CLV} subsequently revisited the
construction through the Grad--Shafranov equation and placed it in
a broader framework that also applies to other hydrodynamic models.

These constructions first produce a local flow satisfying
\begin{equation}\label{intro:eq-localizable}
 u\cdot\nabla p=0.
\end{equation}
This additional identity permits localization: if
$\widetilde u=\chi(p)u$ and $Q'=\chi^2$, then
$(\widetilde u,Q(p))$ is again a stationary solution. Such a flow is
called \emph{localizable}. The cutoff step is simple; what is less
immediate is why a suitable local flow satisfying the additional
constraint \eqref{intro:eq-localizable} exists. This constraint had already been studied by Palumbo \cite{Palumbo} in the context of axisymmetric
magnetohydrostatic equilibria. There it is known as the \emph{isodynamic condition}, which in Euler variables takes the equivalent form $u\cdot\nabla |u|^2=0$. We refer to~\cite{PSS} for a detailed discussion of the relation between Palumbo's work and Gavrilov's construction.

Our aim is to develop a different approach that, we hope, makes the
existence of smooth compactly supported flows more transparent. We perturb radial profiles by solving
elliptic equations on fixed domains, choosing the equations so that
constant regions and smooth matching follow directly from comparison
and regularity estimates. No pressure cutoff or description of the
boundary of the support is needed. A further advantage is that
essentially the same analytic framework produces smooth compactly
supported flows on the plane that are not locally radial.

In three dimensions, our construction not only avoids imposing
\eqref{intro:eq-localizable}, but produces flows for which it fails.

\begin{theorem}\label{thm:spatial}
There exist $u\in C_c^\infty(\RR^3;\RR^3)$ and
$p\in C_c^\infty(\RR^3)$ solving \eqref{eq:euler} with
\mbox{$u\cdot\nabla p\not\equiv0$}.
\end{theorem}

The velocities are axisymmetric with nonzero swirl and supported in
toroidal shells of large major radius; see
Theorem~\ref{thm:spatial-precise}. Thus, the distinction from the
localizable constructions \cite{CLV, Gavrilov} concerns the pressure along streamlines,
not axial symmetry. Nonlocalizable compactly supported flows were
previously obtained in~\cite{DEPS}, with piecewise smooth velocities
discontinuous across a surface. We refer as well to \cite{PSS} for a complementary rigidity theorem for localizable flows under interior analyticity and boundary hypotheses on the Bernoulli function.

In two dimensions, a smooth compactly supported radial stream function
$\psi$ gives a stationary flow
$u=\nabla^\perp\psi:=(\partial_2\psi,-\partial_1\psi)$, as do
superpositions of such flows with disjoint supports and different
centers. These examples are \emph{locally radial}: on each connected
component of $\{\nabla\psi\ne0\}$, the stream function is radial
about a fixed center. Our planar construction goes beyond this class.

\begin{theorem}\label{thm:planar}
There exist $u\in C_c^\infty(\RR^2;\RR^2)$ and
$p\in C_c^\infty(\RR^2)$ solving \eqref{eq:euler} such that $u$ is not
locally radial.
\end{theorem}

Theorem~\ref{thm:planar-precise} gives a more detailed statement of
Theorem~\ref{thm:planar}. In particular, we construct a sequence of
such flows converging in every $C^q$ norm to a radial flow, with all
velocity supports contained in a fixed compact annulus. These are
the first examples of compactly supported planar Euler flows of
class $C^\infty$ that are not locally radial. Applying the localization
of rigid rotation from~\cite{EFRrot}, we also obtain smooth uniformly
rotating flows; see Remark~\ref{p:rem:rotating}.

Such flows were previously constructed at lower regularity.
At the level of weak solutions, the convex integration method of
Choffrut and Sz\'ekelyhidi~\cite{CS} yields compactly supported
$L^\infty$ stationary Euler flows in the plane that are not locally
radial. G\'omez-Serrano, Park and Shi~\cite{GPS} subsequently
constructed such flows of vortex patch type, with continuous
velocities that are not $C^1$. In joint work with Sicbaldi~\cite{EFRS},
we obtained examples of the same regularity through a Schiffer-type
problem on annuli. Finally, in~\cite{EFR} we constructed examples
of class $C^k$ for every prescribed finite $k$. Here each constructed
solution is itself $C^\infty$. The present construction uses a compact
inverse on a fixed domain in place of the derivative-loss framework
of~\cite{EFR}, and does not track a moving support.

Several rigidity results explain the role of the geometry of the
stagnation set. Sharafutdinov's classification~\cite{Sharafutdinov}
implies that smooth compactly supported planar localizable flows
are locally radial. Hamel and Nadirashvili~\cite{HN} proved circularity
for smooth flows tangent to the boundary of a circular annulus and
having no stagnation points in the closed annulus. For smooth
compactly supported flows, Ruiz~\cite[Theorem~A]{Ruiz}
showed that the flow is circular if its non-stagnation set is an
annular domain with continuous boundary. More recently, De Regibus,
Esposito and Ruiz~\cite{DER} proved circularity for
classical flows with finite energy and connected stagnation set,
under the additional assumption that the vorticity has no local
extrema outside a ball.

Figure~\ref{fig:flow-geometry} illustrates the reference toroidal layer
and the angular deformation of interior planar level curves.

\begin{figure}[!b]
\centering
\begingroup
\definecolor{ceFigColor0}{HTML}{25343F}
\definecolor{ceFigColor1}{HTML}{86b1c9}
\definecolor{ceFigColor2}{HTML}{85b0c8}
\definecolor{ceFigColor3}{HTML}{88b3cb}
\definecolor{ceFigColor4}{HTML}{8ab5cd}
\definecolor{ceFigColor5}{HTML}{82adc5}
\definecolor{ceFigColor6}{HTML}{83aec6}
\definecolor{ceFigColor7}{HTML}{8eb9d1}
\definecolor{ceFigColor8}{HTML}{8cb7cf}
\definecolor{ceFigColor9}{HTML}{87b2ca}
\definecolor{ceFigColor10}{HTML}{8bb6ce}
\definecolor{ceFigColor11}{HTML}{80abc3}
\definecolor{ceFigColor12}{HTML}{8fbad2}
\definecolor{ceFigColor13}{HTML}{81acc4}
\definecolor{ceFigColor14}{HTML}{7faac2}
\definecolor{ceFigColor15}{HTML}{92bdd5}
\definecolor{ceFigColor16}{HTML}{91bcd4}
\definecolor{ceFigColor17}{HTML}{89b4cc}
\definecolor{ceFigColor18}{HTML}{8db8d0}
\definecolor{ceFigColor19}{HTML}{7ea9c1}
\definecolor{ceFigColor20}{HTML}{94bfd7}
\definecolor{ceFigColor21}{HTML}{7199ac}
\definecolor{ceFigColor22}{HTML}{7098ab}
\definecolor{ceFigColor23}{HTML}{729aad}
\definecolor{ceFigColor24}{HTML}{739bae}
\definecolor{ceFigColor25}{HTML}{6e96a9}
\definecolor{ceFigColor26}{HTML}{6f97aa}
\definecolor{ceFigColor27}{HTML}{84afc7}
\definecolor{ceFigColor28}{HTML}{7da8c0}
\definecolor{ceFigColor29}{HTML}{769eb1}
\definecolor{ceFigColor30}{HTML}{759db0}
\definecolor{ceFigColor31}{HTML}{96c1d9}
\definecolor{ceFigColor32}{HTML}{95c0d8}
\definecolor{ceFigColor33}{HTML}{749caf}
\definecolor{ceFigColor34}{HTML}{6d95a8}
\definecolor{ceFigColor35}{HTML}{6c94a7}
\definecolor{ceFigColor36}{HTML}{779fb2}
\definecolor{ceFigColor37}{HTML}{79a1b4}
\definecolor{ceFigColor38}{HTML}{78a0b3}
\definecolor{ceFigColor39}{HTML}{98c3db}
\definecolor{ceFigColor40}{HTML}{93bed6}
\definecolor{ceFigColor41}{HTML}{6b93a6}
\definecolor{ceFigColor42}{HTML}{7ba6be}
\definecolor{ceFigColor43}{HTML}{7aa2b5}
\definecolor{ceFigColor44}{HTML}{7ba3b6}
\definecolor{ceFigColor45}{HTML}{9ac5dd}
\definecolor{ceFigColor46}{HTML}{99c4dc}
\definecolor{ceFigColor47}{HTML}{7ca4b7}
\definecolor{ceFigColor48}{HTML}{97c2da}
\definecolor{ceFigColor49}{HTML}{9bc6de}
\definecolor{ceFigColor50}{HTML}{7aa5bd}
\definecolor{ceFigColor51}{HTML}{7ea6b9}
\definecolor{ceFigColor52}{HTML}{7da5b8}
\definecolor{ceFigColor53}{HTML}{6a92a5}
\definecolor{ceFigColor54}{HTML}{7ca7bf}
\definecolor{ceFigColor55}{HTML}{7fa7ba}
\definecolor{ceFigColor56}{HTML}{9dc8e0}
\definecolor{ceFigColor57}{HTML}{90bbd3}
\definecolor{ceFigColor58}{HTML}{80a8bb}
\definecolor{ceFigColor59}{HTML}{6991a4}
\definecolor{ceFigColor60}{HTML}{9ec9e1}
\definecolor{ceFigColor61}{HTML}{81a9bc}
\definecolor{ceFigColor62}{HTML}{79a4bc}
\definecolor{ceFigColor63}{HTML}{a0cbe3}
\definecolor{ceFigColor64}{HTML}{82aabd}
\definecolor{ceFigColor65}{HTML}{9fcae2}
\definecolor{ceFigColor66}{HTML}{9cc7df}
\definecolor{ceFigColor67}{HTML}{6890a3}
\definecolor{ceFigColor68}{HTML}{a1cce4}
\definecolor{ceFigColor69}{HTML}{83abbe}
\definecolor{ceFigColor70}{HTML}{78a3bb}
\definecolor{ceFigColor71}{HTML}{a2cde5}
\definecolor{ceFigColor72}{HTML}{84acbf}
\definecolor{ceFigColor73}{HTML}{678fa2}
\definecolor{ceFigColor74}{HTML}{a3cee6}
\definecolor{ceFigColor75}{HTML}{85adc0}
\definecolor{ceFigColor76}{HTML}{77a2ba}
\definecolor{ceFigColor77}{HTML}{a4cfe7}
\definecolor{ceFigColor78}{HTML}{a5d0e8}
\definecolor{ceFigColor79}{HTML}{76a1b9}
\definecolor{ceFigColor80}{HTML}{a6d1e9}
\definecolor{ceFigColor81}{HTML}{86aec1}
\definecolor{ceFigColor82}{HTML}{668ea1}
\definecolor{ceFigColor83}{HTML}{a7d2ea}
\definecolor{ceFigColor84}{HTML}{87afc2}
\definecolor{ceFigColor85}{HTML}{a8d3eb}
\definecolor{ceFigColor86}{HTML}{75a0b8}
\definecolor{ceFigColor87}{HTML}{a9d4ec}
\definecolor{ceFigColor88}{HTML}{88b0c3}
\definecolor{ceFigColor89}{HTML}{E0B997}
\definecolor{ceFigColor90}{HTML}{79848D}
\definecolor{ceFigColor91}{HTML}{E9EFF3}
\definecolor{ceFigColor92}{HTML}{FFFFFF}
\definecolor{ceFigColor93}{HTML}{B6C2CB}
\definecolor{ceFigColor94}{HTML}{24628A}

\endgroup
\caption{Schematic reference geometry. (a) A toroidal layer obtained by
rotating a meridional annulus about the symmetry axis. (b) An angular
deformation of regular planar level curves in an interior annular band
$B$, with $\ell=5$ for illustration; arrows indicate velocity.
Dashed circles mark reference levels and the edges of $B$.
Neither panel depicts the actual free boundary.}
\label{fig:flow-geometry}
\end{figure}
\vspace{-0.25cm}
\subsection{Strategy of the proof}
The starting point is a radial stream function that is constant outside
an annulus. In three dimensions we place its meridional profile far
from the symmetry axis and correct the small error introduced by
curvature. In two dimensions we vary the radius of the annulus until
an angular mode bifurcates. The common problem is to carry out these
perturbations while preserving exact constant regions and smooth
contact with them. The proof combines two shared analytic ideas
with two different geometric applications.

\vspace{-0.25cm}
\subsubsection*{First idea: A priori regularity at the contact set}
Elliptic equations with non-Lipschitz nonlinearities can admit
nontrivial solutions that vanish on open sets~\cite[Section~6]{BBC}.
We first establish conditions under which such solutions
are smooth, with all derivatives vanishing on their zero sets.

Consider a nonnegative, locally bounded weak solution $v$ to
\[
 \Delta v+b(X)\cdot\nabla v=\mathcal F(X,v),
\]
where $b$ is smooth and $\mathcal F$ is continuous up to $t=0$,
smooth for $t>0$, and satisfies $\mathcal F(X,0)=0$. Assume that, locally uniformly in $X$,
\begin{equation}\label{intro:eq-near-linear}
 |D_X^\alpha\partial_t^j\mathcal F(X,t)|
 \le C_{\alpha,j,\eta}t^{1-j-\eta},\qquad 0<t\ll1,
\end{equation}
for every multi-index $\alpha$, every integer $j\ge0$ and every
$\eta>0$. The constants in these estimates may depend on $\eta$.
Theorem~\ref{common:lem-flat-regularity} shows that $v$ is smooth
and all its derivatives vanish wherever $v=0$, without any
assumption on the geometry of the zero set.
 
The proof has two parts. First, Harnack's inequality and the bounds
for $\mathcal F$ imply that $v$ vanishes to every order at each of
its zeros. More precisely, near any point $X_0$ with $v(X_0)=0$,
\[
 0\le v(X)\le C_N|X-X_0|^N,\qquad N\ge1.
\]
We then apply interior elliptic estimates on balls contained in
$\{v>0\}$, with radii chosen in terms of the value of $v$ at their
centers. The derivative bounds in~\eqref{intro:eq-near-linear}
yield corresponding estimates for all derivatives of $v$, which
give smoothness across the zero set. The nonlinearities arising from our reference profiles satisfy
\eqref{intro:eq-near-linear}.

In addition, comparison barriers ensure that nearby stream functions attain their prescribed constant values in neighborhoods of the boundary components of
the fixed domain (Lemma~\ref{common:lem-exact-collar}).
We can therefore extend them by these constants outside the
domain. The regularity theorem, Theorem \ref{common:lem-flat-regularity}, guarantees smoothness wherever a constant value is attained, without requiring us to determine
the boundaries of the constant regions. Section~\ref{common:sec-prelim}
develops these regularity and localization tools.

\vspace{-0.25cm}
\subsubsection*{Second idea: A compact fixed-point formulation} The derivative of the nonlinearity is unbounded at the constant states, so we do not linearize the original equation directly. Instead, on a fixed domain $D$, we rewrite the equation as
\[
 -L\psi+g(X,\psi)=k(X,\psi), \qquad \psi=\mathcal T[k(\cdot,\psi)],
\]
where $k$ is smooth, $t \mapsto g(X,t)-t$ is nondecreasing for each $X \in D$, and $\mathcal{T}$ is the inverse of $-L+g$ with the prescribed boundary data. Comparison gives continuity of this fixed-point map, while
elliptic estimates give compactness. To determine its first-order behavior, we study normalized differences of nearby solutions.
Proposition~\ref{common:prop-linearization} approximates these
differences by solutions to a linear Dirichlet problem with a
singular potential on the reference annulus, extended by zero
to the constant regions. These expansions allow us to apply
Leray--Schauder degree without requiring $\mathcal T$ to be
$C^1$ on a neighborhood of the reference source. Section~\ref{common:sec-compact}
develops this compact formulation.

\vspace{-0.25cm}
\subsubsection*{The case of 3$d$ fluids} We place the radial reference profile in the meridional plane
at distance $R$ from the symmetry axis. Note that it satisfies the
axisymmetric equation up to an error of order $R^{-1}$.

The kernel of the limiting linearized operator $(R=\infty)$ consists of
the two translation modes
(Proposition~\ref{s:prop:spec-planar-kernel}).
Reflection symmetry excludes the vertical mode, and we require
the correction to be orthogonal to the remaining mode.
Solvability then requires the leading curvature error to be
orthogonal to this mode as well. We achieve this by adjusting
the parameter controlling the azimuthal velocity
(Lemma~\ref{s:lem:Melnikov}). A topological degree argument then gives solutions for all sufficiently large $R$.

To prove nonlocalizability, we denote by $u$ the constructed solution and expand $|u|^2$ along level curves
of the corresponding stream function. The first-order angular term cannot
vanish identically, since this would force the reference profile
to satisfy a smooth first-order ODE with initial value at an
equilibrium, and uniqueness would then imply that the profile is
constant, giving a contradiction. Consequently, for all sufficiently
large $R$, $|u|^2$ is nonconstant along at least one level curve.
Since the Bernoulli function is constant on this curve, the
pressure is nonconstant there, and hence
the solution is not localizable (Proposition~\ref{s:prop:nl-euler-field}).
Section~\ref{s:sec:profile} contains the three-dimensional construction.

\vspace{-0.25cm}
\subsubsection*{The case of 2$d$ fluids}
The planar reference profile increases from zero and then decreases
back to zero. In this setting, a single global equation
$\Delta\psi=f(\psi)$ would force local radiality by the symmetry
results of~\cite{Brock} (see~\cite[Theorem~A.2]{EFR}). Following
\cite{EFR}, we use separate nonlinearities on the two profile branches
and interpolate between them where both prescribe the same
constant vorticity. For nearby solutions, the vorticity is then
locally a function of the stream function, which ensures the
stationary two-dimensional Euler equation (Lemma~\ref{p:lem:autonomy}).

The radial derivative changes sign once, giving a negative
eigenvalue in the first angular mode. In angular mode $n$, the
linearized operator contains the positive term $n^2/r^2$.
We choose $n$ sufficiently large so that the lowest eigenvalue in
this mode is positive at a fixed annulus radius. Keeping $n$
fixed and increasing the radius eventually makes this eigenvalue
negative. In a suitable symmetry class, all other eigenvalues
remain separated from zero (Proposition~\ref{p:prop:spectral-input}). The resulting change of degree produces nonradial solutions converging to a radial
profile (Proposition~\ref{p:prop:bifurcation}). Their limiting angular mode in \eqref{p:eq:tangent}, together with their discrete
rotational symmetry, rules out local radiality. Section~\ref{p:sec-reference} contains the planar construction.

In both applications, exact collars and a priori regularity complete
the construction of smooth compactly supported flows.

\section{A priori regularity and exact localization}\label{common:sec-prelim}

We begin with two key analytic results used in both constructions.
First, we give conditions on the nonlinearity under which
nonnegative, locally bounded weak solutions are smooth and all
their derivatives vanish on their zero sets. We then introduce
a flat reference profile and use comparison arguments to obtain
a criterion for exact vanishing near the boundary. These results
will allow us to extend the stream functions smoothly by their
prescribed constant values, without determining the boundaries
of their constant regions.
 
\subsection{A priori regularity across the zero set}
Standard elliptic estimates give smoothness where the solution is
positive; the issue is regularity at its zeros. The model $v''=v^p$,
$0<p<1$, admits contact of finite order $2/(1-p)$. The hypothesis below
is available for every exponent below one and yields infinite-order
contact, without assumptions on the geometry of the zero set.

\begin{theorem}\label{common:lem-flat-regularity}
Let $V\subset\RR^2$ be open, let $b\in C^\infty(V;\RR^2)$, and let
\[
 \mathcal F\in C(V\times[0,\infty))
 \cap C^\infty(V\times(0,\infty)),
 \qquad \mathcal F(X,0)=0.
\]
Assume that, for every compact set $K\Subset V$, every $M>0$,
every multi-index $\alpha$, every integer $j\ge0$, and every
$\eta>0$, there is a constant $C = C(K,M,\alpha,j,\eta) > 0$ such that
\begin{equation}\label{common:eq-nearlinear}
 |D_X^\alpha\partial_t^j\mathcal F(X,t)|
 \le C t^{1-j-\eta},
 \qquad X\in K,\quad 0<t\le M.
\end{equation}
Let $v\in H^1_{\mathrm{loc}}(V)\cap L^\infty_{\mathrm{loc}}(V)$
be nonnegative almost everywhere and satisfy
\[
 \Delta v+b(X)\cdot\nabla v=\mathcal F(X,v)
 \quad\text{weakly in }V.
\]
Then $v\in C^{\infty}(V)$ 
and all its derivatives vanish on $Z:=\{X\in V:v(X)=0\}$.
If $Z\ne\emptyset$, then for every compact $K\Subset V$ and
integers $q,N\ge0$ there are constants $C = C(K,q,N)>0$ and
$\delta = \delta(K,q,N)>0$ such that
\begin{equation}\label{common:eq-jets-flat}
 |D^qv(X)|\le C \operatorname{dist}(X,Z)^N,
 \qquad X\in K\setminus Z,\quad
 \operatorname{dist}(X,Z)<\delta.
\end{equation}
\end{theorem}

\begin{proof}
Fix $K\Subset V$ and $0<\rho\le1/8$ so that $U_8\Subset V$, where
$U_j:=\{X:\operatorname{dist}(X,K)<j\rho\}$, $1\le j\le8$,
choose $M\ge\max\{1,\|v\|_{L^\infty(U_8)}\}$, and note that the constants below may depend on these choices and the coefficient bounds on $\overline U_8$. We split the rest of the proof into three parts. 

\step{1. Additive Harnack and infinite-order vanishing.}
The forcing $\mathcal F(\cdot,v)$ is bounded. The local energy estimate,
followed by interior $W^{2,p}$ regularity in dimension two, gives
\[
 \|v\|_{W^{2,p}(U_4)}
 \le C_p\bigl(M+\|\mathcal F(\cdot,v)\|_{L^\infty(U_8)}\bigr),
 \qquad 1<p<\infty,
\]
where $C_p$ depends only on $p$, the neighborhoods and the drift
bound; see~\cite{GT}. Thus $v\in C^{1,\gamma}_{\loc}(V)$ for
$0<\gamma<1$. Its continuous representative is nonnegative, and
interior bootstrapping gives $v\in C^\infty(V_+)$ on $V_+:=\{v>0\}$.

Fix $0<\theta<1$. By~\eqref{common:eq-nearlinear},
$|\mathcal F(X,t)|\le Ct^{1-\theta}$ on $\overline U_4\times[0,M]$.
For $A\in U_2$ and $0<r\le\rho$, put $k=r^{2/\theta}$ and $w=v+k$.
On $B_{2r}(A)\subset U_4$, it follows that
\[
 \Delta w+b\cdot\nabla w=c_rw,
 \qquad c_r(X)=\frac{\mathcal F(X,v(X))}{v(X)+k}.
\]
Moreover, the inequality
\[
 \frac{a^{1-\theta}}{a+k}
 =k^{-\theta}\frac{(a/k)^{1-\theta}}{1+a/k}
 \le k^{-\theta},\qquad a\ge0,
\]
gives $\|c_r\|_{L^\infty(B_{2r}(A))}\le Cr^{-2}$. Rescaling by $X=A+rY$ makes
the drift $rb(A+rY)$ and the potential $r^2c_r(A+rY)$ uniformly
bounded on $B_2$. Harnack's inequality~\cite{GT} therefore yields
\begin{equation}\label{common:eq-additive-harnack}
 \sup_{B_r(A)}(v+r^{2/\theta})
 \le C_H\inf_{B_r(A)}(v+r^{2/\theta}),
\end{equation}
with $C_H$ independent of $A$ and $r$.
At a zero $A\in Z \cap U_2$, this gives $\sup_{B_r(A)}v\le C_Hr^{2/\theta}$.

For $X\in K\setminus Z$, set $d:=\operatorname{dist}(X,Z)<\rho/4$,
choose a nearest zero $A\in\overline U_1\Subset V$ and take $r=2d$.
Since $\theta$ can be arbitrarily small, for every integer $N\ge0$, it follows that
\begin{equation}\label{common:eq-value-flat}
 v(X)\le C_N\operatorname{dist}(X,Z)^N,
 \qquad X\in K\setminus Z,\quad
 \operatorname{dist}(X,Z)<\rho/4.
\end{equation}

\step{2. Derivative estimates at the positive-value scale.}
At $X\in K$ with $t:=v(X)>0$, choose a sufficiently small fixed
$\kappa>0$, depending on $C_H$, and put $r=(\kappa t)^{\theta/2}$.
For small $t$, $r\le\rho$ and~\eqref{common:eq-additive-harnack} gives
\[
 c_*t\le v(Y)\le C_*t\quad\text{on }B_r(X),
\]
with fixed $c_*,C_*>0$. In particular, this ball lies in $V_+$ so $v$ is smooth on this ball. Next, set
\[
 u(Y)=\frac{v(X+rY)}t,\qquad \beta(Y)=rb(X+rY),\qquad
 G(Y,s)=\frac{r^2}{t}\mathcal F(X+rY,ts).
\]
Then $\Delta u+\beta\cdot\nabla u=G(Y,u)$ on $B_1$, with
$c_*\le u\le C_*$. On $I=[c_*/2,2C_*]$, for small $t$ such that
$2C_*t\le M$, the hypothesis with $\eta=\theta$ gives
\begin{align*}
 |D_Y^\alpha\partial_s^jG(Y,s)|
 &=r^{2+|\alpha|}t^{j-1}
   |D_X^\alpha\partial_t^j\mathcal F(X+rY,ts)|\\
 &\le C r^{2+|\alpha|}t^{-\theta}
  =C\kappa^\theta r^{|\alpha|}\le C.
\end{align*}
Here $s^{1-j-\theta}$ is uniformly bounded on $I$. Similarly,
$D^\alpha\beta=r^{1+|\alpha|}(D^\alpha b)(X+rY)$ is uniformly bounded.
Interior $W^{2,p}$ estimates and Schauder bootstrapping on $B_1$
therefore yield $\|u\|_{C^q(B_{1/2})}\le C_q$, independently of $X,t$.
Scaling back gives
\[
 |D^qv(X)|\le C_qtr^{-q}
 =C_q\kappa^{-q\theta/2}t^{1-q\theta/2}.
\]
Taking $\theta=1/(q+2)$, and restricting to $t<t_q<1$, proves
\begin{equation}\label{common:eq-derivative-small-value}
 |D^qv(X)|\le C_qv(X)^{1/2},
 \qquad X\in K,\quad 0<v(X)<t_q.
\end{equation}
The constants and threshold depend only on $q$ and the fixed data.

\step{3. Extension of the jets.}
The order-two bound in~\eqref{common:eq-value-flat} makes $v<t_q$
sufficiently near $Z$. Combining~\eqref{common:eq-derivative-small-value}
with~\eqref{common:eq-value-flat} at order $2N$ then gives
\eqref{common:eq-jets-flat} on $V_+$. We now show that these derivatives extend as the
classical derivatives of $v$ on $V$. Define
\[
 J_\alpha(X)=
 \begin{cases}
 D^\alpha v(X),&X\in V_+,\\
 0,&X\in Z.
 \end{cases}
\]
The estimate with $N=1$ makes $J_\alpha$ continuous. With $N=2$, near
any $X_0\in Z$ it gives
\[
 |J_\alpha(X)-J_\alpha(X_0)|
 \le C\operatorname{dist}(X,Z)^2\le C|X-X_0|^2.
\]
Thus $J_\alpha$ is differentiable at $X_0$ with zero differential.
On $V_+$, and hence also on $Z$, its partial derivatives are the
continuous functions $J_{\alpha+e_i}$. Starting from $J_0=v$,
induction proves $v\in C^\infty(V)$ and $D^\alpha v=J_\alpha$.
All derivatives vanish on $Z$.
\end{proof}

\begin{remark}
By smoothness for positive values, assumption~\eqref{common:eq-nearlinear}
can equivalently be required only for sufficiently small positive
$t$, with a threshold depending on $K,\alpha,j,\eta$.
On the other hand, if $|\mathcal F(X,t)|\le Ct$ near zero, locally uniformly in $X$, Harnack's inequality implies that a nonnegative solution on a
connected domain is either positive everywhere or identically zero.
The nonlinearities used below satisfy~\eqref{common:eq-nearlinear}
while permitting nontrivial solutions to vanish on open sets.
\end{remark}

\begin{remark}
Let $(b_\lambda,\mathcal F_\lambda,v_\lambda)$ be a family
satisfying the hypotheses of Theorem~\ref{common:lem-flat-regularity}
on the same open set $V$. Assume that, for every compact
$K\Subset V$ and integer $m\ge0$,
\[
 \sup_\lambda\bigl(
 \|v_\lambda\|_{L^\infty(K)}+\|b_\lambda\|_{C^m(K)}\bigr)<\infty,
\]
and that the constant in~\eqref{common:eq-nearlinear} can be
chosen independently of $\lambda$ for every fixed
$K,M,\alpha,j,\eta$. Then
\[
 \sup_\lambda\|v_\lambda\|_{C^q(K)}<\infty,
 \qquad K\Subset V,\quad q\ge0.
\]
The constants $C$ and $\delta$ in~\eqref{common:eq-jets-flat}
can also be chosen independently of $\lambda$, with $Z$ replaced
by the zero set $Z_\lambda$ of each member for which
$Z_\lambda\ne\emptyset$.

Indeed, the local estimate in part~1 gives a uniform $C^1$ bound
from the stated data, without a separate bound on the initial $H^1$
norms. All constants and small-value thresholds in parts~1--3 are then
uniform. For each fixed $q$, at points where $v_\lambda\ge t_q$, the
uniform gradient bound supplies a common ball on which
$v_\lambda\ge t_q/2$; ordinary interior regularity on this positive
value interval bounds all derivatives up to order $q$ uniformly.
Together with~\eqref{common:eq-derivative-small-value} and flatness
at zeros, this proves the uniform $C^q$ bound, also when the zero set
is empty. Finally, the uniform order-two value bound gives a common
distance threshold for $v_\lambda<t_q$, and the order-$2N$ value bound
then gives the asserted uniform constants in~\eqref{common:eq-jets-flat}.
\end{remark}

\subsection{A flat profile and exact boundary collars}\label{common:sec-collars}

We need equations that allow constant regions, but that at the same time satisfy the hypotheses of the regularity theorem we just proved. We construct them from a single
increasing profile $H$ that is flat at zero: expressing its derivatives
as functions of its value gives the required nonlinearities. Shifted
copies of $H$ will also provide barriers that force sufficiently small
solutions with zero boundary values to vanish in a whole boundary collar.
For $0<s<1$, set
\begin{equation}\label{common:eq-H}
 H(s):=e^{-L(s)^2},\qquad L(s):=\log(1/s),
\end{equation}
and extend it by zero for $s\le0$. Differentiation gives
$H^{(n)}(s)=s^{-n}H(s)P_n(L(s))$ with $P_n$ a polynomial.
The factor $e^{-L^2}$ dominates every power of $s^{-1}$ and $L$,
so this extension is smooth and flat at zero. The identities used
in the collar and spectral estimates are
\begin{align}
 \frac{H''(s)}{H'(s)}
 &=\frac{2L(s)-1-L(s)^{-1}}s,
       \label{common:eq-Hratio2}\\
 \frac{H'''(s)}{H'(s)}
 &=\frac{4L(s)^2-6L(s)-4+3/L(s)}{s^2}.
       \label{common:eq-Hratio3}
\end{align}
Note as well that, given $0<c<4$, we can fix $s_c=s_c(c)\in(0,1)$
so small that $H'$ and $H''$ are positive on $(0,s_c)$ and
\[
 c\,\frac{\log^2(1/s)}{s^2}
 \le\frac{H'''(s)}{H'(s)}
 \le C\,\frac{\log^2(1/s)}{s^2},
 \qquad 0<s<s_c.
\]

Writing $M(t)=\log(1/t)$, we have $H^{-1}(t)=e^{-\sqrt{M(t)}}$.
Moreover, setting $F=H''\circ H^{-1}$, we have
\[
 F(t)=t e^{2\sqrt{M(t)}}(4M(t)-2\sqrt{M(t)}-2) \quad \textup{and} \quad
 F'(t)\sim4M(t) e^{2\sqrt{M(t)}}\quad \textup{as } t\downarrow0.
\]
Its derivative is unbounded, but $e^{2\sqrt M}$ grows more slowly
than $e^{\eta M}$ for every $\eta>0$. Moreover, at the reference profile $t=H(d)$, it gives
$F'(H(d))\approx\log^2(1/d)/d^2$, where $d$ is distance to a
reference interface. This is precisely the behavior of the ``linearized operator'' we will have to handle. We also record the scale comparison
\begin{equation}\label{common:eq-inverse-scale}
 H^{-1}\bigl(2H(2\delta)\bigr)\le3\delta
 \qquad\textup{for }0<\delta\ll1.
\end{equation}
Indeed, with $\ell=\log(1/(2\delta))$, the ratio to $2\delta$ is
$\exp(\ell-\sqrt{\ell^2-\log2})\to1$ as $\delta \to 0$.

More generally, for $A,B\in C^\infty([0,s_c])$, the value law
\[
 F(t)=A(H^{-1}(t))H''(H^{-1}(t))
       +B(H^{-1}(t))H'(H^{-1}(t)),\qquad 0<t<H(s_c),
\]
extends continuously to $F(0)=0$ and satisfies, for every integer $j\ge0$ and
$\eta>0$,
\begin{equation}\label{common:eq-value-calculus}
 |F^{(j)}(t)|\le C_{j,\eta}t^{1-j-\eta}.
\end{equation}
These bounds are uniform when the derivatives of $A,B$ are uniformly
bounded. They also hold after differentiation in additional spatial
or scalar parameters with uniformly bounded mixed derivatives.
Indeed, it follows that $F(t)=tQ(M(t))$, where
\[
 Q(M)=A(e^{-\sqrt M})e^{2\sqrt M}(4M-2\sqrt M-2)
       +2B(e^{-\sqrt M})e^{\sqrt M}\sqrt M.
\]
Since $M\ge\log(1/H(s_c))>0$, the chain rule gives
$|Q^{(k)}(M)|\le C_ke^{2\sqrt M}(1+M)^{N_k}$ for every fixed $k$.
Thus, taking into account that $\partial_t=-t^{-1}\partial_M$, it follows that
\[
 F^{(j)}(t)=t^{1-j}Q_j(M(t)),\qquad
 Q_0=Q,\qquad Q_{j+1}=(1-j)Q_j-Q_j'.
\]
Hence, \eqref{common:eq-value-calculus} follows from
$e^{2\sqrt M}(1+M)^N\le C_{N,\eta}e^{\eta M}$.
Taking $j=0$ and $\eta<1$ proves continuity at zero. Parameter
derivatives replace $A,B$ by their derivatives in this same calculation.

We next use a translated and rescaled copy of $H$ to force exact
vanishing near a prescribed boundary circle. Fix $X_c\in\RR^2$ and $r_0>0$, and write
\[
 \Gamma=\{X\in\RR^2:|X-X_c|=r_0\},\qquad
 r(X)=|X-X_c|,\qquad e_r(X)=\frac{X-X_c}{r(X)}.
\]
Choose $\epsilon\in\{-1,1\}$ and set $s(X)=\epsilon(r(X)-r_0)$.
The side $s>0$ is the side on which the equation is considered.
At a boundary component of an annulus $\{r_- < r < r_+\}$ we choose this side inward:
$s=r-r_-$ at the inner boundary and $s=r_+-r$ at the outer
boundary. Thus, $s$ increases as we move from either boundary into
the domain. 

For $0<h\le s_*<r_0$, define
\[
 \mathcal C_h=\{X:0<s(X)<h\},\qquad
 \widetilde{\mathcal C}_h=\{X:|s(X)|<h\},\qquad
 \Gamma_t=\{X:s(X)=t\},\quad |t|\le s_*.
\]
The one-sided collar $\mathcal C_h$ is the open annulus between
$\Gamma_0=\Gamma$ and $\Gamma_h$; the circle $\Gamma_t$ has radius
$r_0+\epsilon t$. The two-sided collar
$\widetilde{\mathcal C}_h$ includes both sides of $\Gamma$.
These collars are taken around prescribed circles, independently
of the support of a solution. When both boundary components of an
annulus are considered, their collars are chosen disjoint.

Since $s_*<r_0$, the coordinate $s$ is smooth on the closed collars
and satisfies
\[
 \nabla s=\epsilon e_r,\qquad |\nabla s|=1,\qquad
 \Delta s=\frac{\epsilon}{r}.
\]
For $\mathcal L=\Delta+b\cdot\nabla$, with $b$ smooth up to the
boundary of the collar, the chain rule therefore gives
\[
 \mathcal L(w\circ s)=w''(s)+\beta(X)w'(s),\qquad
 \beta(X)=\epsilon\left(\frac1r+b(X)\cdot e_r\right).
\]
On $\mathcal C_{s_*}$, both $|\beta|$ and $|b|$ are bounded by
any constant $B$ satisfying
\[
 B\ge\frac1{r_0-s_*}
       +\|b\|_{L^\infty(\mathcal C_{s_*};\RR^2)}.
\]
The same bound holds on $\widetilde{\mathcal C}_{s_*}$ when $b$
is defined there, with its norm taken on that collar. In either
case, $B$ can be chosen uniformly for families if $r_0-s_*$ has
a positive lower bound and the drifts are uniformly bounded.
We abbreviate $w\circ s$ by $w(s)$. Neither $\beta$ nor the
solutions considered below are assumed to be radial.

The lemma below concerns exact vanishing at the artificial boundary of $D$. The linear
barriers in Section~\ref{common:sec-compact} will instead give smallness
near the reference interfaces of $\Omega_\tau$. Neither argument
describes the actual free boundary.

\begin{lemma}
\label{common:lem-exact-collar}
Let $\mathcal C_{s_*}$ be the one-sided collar defined above, and
suppose that
\[
 \mathcal L=\mu^{-1}\Div(\mu\nabla)
           =\Delta+\nabla\log\mu\cdot\nabla,
\]
where $\mu$ is smooth and strictly positive on
$\overline{\mathcal C}_{s_*}$. Let $B\ge0$ satisfy
$|\mathcal Ls|=|\beta|\le B$ on this collar. Given $t_0>0$, let
\[
 A\in C(\overline{\mathcal C}_{s_*}\times[0,t_0]),\qquad
 A(X,0)=0,
\]
and assume that $t\mapsto A(X,t)$ is nondecreasing for each $X$.
Moreover, suppose that there are $c_0>0$ and $\sigma_0\in(0,s_c)$ such that
\begin{equation}\label{common:eq-exact-collar-absorption}
 H(\sigma)<t_0,\qquad
 A(X,H(\sigma))\ge c_0H''(\sigma),
 \qquad 0<\sigma<\sigma_0,\quad
 X\in\overline{\mathcal C}_{s_*}.
\end{equation}
Then there exist $\delta\in(0,s_*/2)$ and $\eta_0\in(0,t_0)$
such that every nonnegative weak solution
$v\in H^1(\mathcal C_{2\delta})\cap L^\infty(\mathcal C_{2\delta})$
to
\[
 -\mathcal L v+A(X,v)=0 \quad\text{in }\mathcal C_{2\delta}, 
 \qquad v=0 \quad\text{on }\Gamma_0, \qquad \textup{\textit{with} }\ \|v\|_{L^\infty(\mathcal C_{2\delta})}\le\eta_0,
\]
satisfies
$v=0$ almost everywhere in $\mathcal C_\delta$.
\end{lemma}

\begin{proof}
First of all, we choose $0<\kappa<\min\{1,\sqrt{c_0/2}\}$. Then, since
$H''(\sigma)/H'(\sigma)\to\infty$ as $\sigma \downarrow 0$ by~\eqref{common:eq-Hratio2}, there
exists $\sigma_1\in(0,\sigma_0)$ such that
\[
 \frac{H''(\sigma)}{H'(\sigma)}\ge\frac{2B\kappa}{c_0},
 \qquad 0<\sigma\le\sigma_1.
\]
The positivity of $H'$ and $H''$ on this interval gives
\[
 (c_0-\kappa^2)H''(\sigma)-B\kappa H'(\sigma)\ge0,
 \qquad 0<\sigma\le\sigma_1.
\]
Having this inequality at hand, we fix
\[
 \delta=\frac14\min\{s_*,\sigma_1/\kappa\},\qquad
 \eta_0=H(\kappa\delta),\qquad
 W(X)=H\bigl(\kappa(s(X)-\delta)\bigr).
\]
Then $0<\delta<s_*/2$ and $0<\kappa\delta<\sigma_0$, so
$0<\eta_0<t_0$. Moreover, on $\overline{\mathcal C}_{2\delta}$ we have
$0\le W\le\eta_0$, with $W=0$ for $s\le\delta$ and
$W=\eta_0$ on $\Gamma_{2\delta}$. 

Next, for $\delta<s(X)<2\delta$, we set
$\sigma(X)=\kappa(s(X)-\delta)$. The chain rule and
\eqref{common:eq-exact-collar-absorption} yield
\begin{align*}
 -\mathcal LW+A(X,W)
 &=-\kappa^2H''(\sigma)-\kappa\beta(X)H'(\sigma)
                         +A(X,H(\sigma))\ge(c_0-\kappa^2)H''(\sigma)-B\kappa H'(\sigma)\ge0.
\end{align*}
The zero extension of $H$ is smooth and flat at zero, so $W$ is
smooth across $\Gamma_\delta$. Likewise, for $s\le\delta$, its derivatives
and $A(X,W)=A(X,0)$ vanish. Hence, $W$ is a supersolution on the
whole collar $\mathcal{C}_{2\delta}$, with no interface term at $\Gamma_\delta$.

Now, we set $q=(v-W)_+$ and observe that $q \in H_0^1(\mathcal C_{2\delta})$. It is then immediate to check that
\[
 \int_{\mathcal C_{2\delta}}\mu|\nabla q|^2\dd X
 +\int_{\mathcal C_{2\delta}}\mu[A(X,v)-A(X,W)]q\dd X\le0.
\]
Since both integrands are nonnegative,  it follows that
$\nabla q=0$ almost everywhere, and Poincar\'e's inequality
for $q\in H_0^1(\mathcal C_{2\delta})$ gives $q=0$ almost everywhere.
Consequently, $0\le v\le W$ almost everywhere, which proves
that $v=0$ almost everywhere in $\mathcal C_\delta$.
\end{proof}

\begin{remark}
The constants $\delta$ and $\eta_0$ can be chosen uniformly for families with the same collar, $B$, $c_0$, $\sigma_0$ and $t_0$. Indeed,  $\kappa$ depends only on $c_0$, and $\sigma_1$ can be
chosen using only $B,c_0,\sigma_0$ and the fixed profile $H$.
The choices of $\delta$ and $\eta_0$ are therefore uniform under
the stated family assumptions. No common bound for $\mu$ or
$\mu^{-1}$ is needed for these choices: their boundedness for
each individual member suffices to justify the comparison.
\end{remark}

\section{A compact formulation on a fixed domain}\label{common:sec-compact}

We will need to perturb the reference profiles while allowing their
constant regions to change. The scalar nonlinearities we will consider have unbounded derivatives at the plateau values, so their usual linearization is not a bounded operator on continuous perturbations. We split them into a
strongly increasing part $g$ and a smooth part $k$, and write
\[
 -\mathcal L\psi+g(X,\psi)=k(X,\psi),\qquad
 \psi=\mathcal T[k(X,\psi)].
\]
For a prescribed right-hand side, monotonicity of $g$ gives a unique solution, while elliptic estimates make the inverse $\mathcal T$ compact. This turns
the perturbation problem into a fixed-point equation with a compact operator on a domain larger than the active annulus. 

In this section, we develop the necessary machinery to handle this fixed point equation. We first construct $\mathcal T$, then the linear inverse that describes its first-order response. The large slopes near the plateau values force this response to vanish on the constant regions, leaving a singular linear equation on the reference annulus. The final result identifies this limit without requiring a $C^1$ inverse on an open neighborhood of sources.

Given a bounded smooth domain $D\subset\RR^2$, we write
$C_0(\overline D)$ for the continuous functions on $\overline D$
which vanish on $\partial D$.

\subsection{The nonlinear inverse}

To obtain the monotone equation, we add a smooth correction $k$ to
the scalar nonlinearity $N$ and move the same correction to the
right-hand side. In both applications, $\partial_tN\ge1$ already
holds near the plateau values, so the correction is needed only
away from them. Choose $0\le\xi\le1$, equal to one near each
plateau value and supported where $\partial_tN\ge1$, and choose
$C_g\ge1$ with $C_g\ge\partial_tN$ on the support of $1-\xi$.
On the physical value range, set
\begin{equation}\label{common:eq-monotone-splitting}
 k(X,t)=\int_0^t(1-\xi(s))[C_g-\partial_sN(X,s)]\dd s,
 \qquad g=N+k.
\end{equation}
The integrand is smooth because it vanishes near the singular
values. Thus $k$ is smooth, nondecreasing and constant near each
plateau, while
\[
 \partial_tg=\xi\partial_tN+(1-\xi)C_g\ge1.
\]
Only the extensions outside the physical value range differ between
the two applications. We now construct the inverse for general
continuous strongly monotone nonlinearities.

Fix boundary data $h\in C^\infty(\overline D)$ and a compact metric
parameter space $K$. For $\tau\in K$, let
$\mu_\tau\in C^\infty(\overline D)$ be positive, and set
\[
 \mathcal L_\tau\coloneqq
 \mu_\tau^{-1}\operatorname{div}(\mu_\tau\nabla)
 =\Delta+\nabla\log\mu_\tau\cdot\nabla.
\]
We assume that $\tau\mapsto\mu_\tau$ is continuous in every
$C^m(\overline D)$ norm and that the weights and their reciprocals have
uniform bounds in these norms. In particular, there are constants
$0<c_\mu\le C_\mu<\infty$ such that
\[
 c_\mu\le\mu_\tau(X)\le C_\mu,
 \qquad X\in\overline D,\ \tau\in K.
\]

Let $g_\tau\in C(\overline D\times\RR)$. Parameter continuity means
that, whenever $\tau_j\to\tau$ in $K$, one has
\[
 \sup_{\substack{X\in\overline D\\|t|\le M}}
 |g_{\tau_j}(X,t)-g_\tau(X,t)|\longrightarrow0
 \qquad\text{for every }M<\infty.
\]
Assume as well that, for every $X\in\overline D$, $\tau\in K$ and $s,t\in\RR$, 
\[
 g_\tau(X,0)=0,\qquad
 [g_\tau(X,s)-g_\tau(X,t)](s-t)\ge|s-t|^2,
 \qquad |g_\tau(X,t)|\le C_g(1+|t|),
\]
where $C_g$ is independent of $\tau$. 

\begin{proposition}\label{common:prop-monotone-inverse}
For every $\tau\in K$ and $F\in L^\infty(D)$, the problem
\begin{equation*}
 -\mathcal L_\tau u+g_\tau(X,u)=F,
 \qquad u-h\in H_0^1(D),
\end{equation*}
has a unique weak solution, denoted by $\mathcal T_\tau F$. It belongs to
$W^{2,p}(D)$ for every finite $p$, and
\begin{align}
 \norm{\mathcal T_\tau F}_{L^\infty(D)}
 &\le\max\{\norm h_{C(\partial D)},\norm F_{L^\infty(D)}\},
       \label{common:eq-T-bound}\\
 \norm{\mathcal T_\tau F-\mathcal T_\tau G}_{L^\infty(D)}
 &\le\norm{F-G}_{L^\infty(D)}.\label{common:eq-T-Lipschitz}
\end{align}
The family $\{\mathcal{T}_\tau\}_{\tau \in K}$ is collectively compact from bounded subsets of $L^\infty(D)$ into $C(\overline D)$, and the map
$(\tau,F)\mapsto\mathcal T_\tau F$ is continuous from
$K\times L^\infty(D)$ to $C(\overline D)$.  

If, in addition, there are constants $m\le M$ satisfying $m\le h\le M$ on $\partial D$ and
\[
 g_\tau(X,m)\le F(X)\le g_\tau(X,M)
 \quad\text{almost everywhere in }D,
\]
then $m\le\mathcal T_\tau F\le M$. In particular,
$\mathcal T_\tau F\ge0$ when $F\ge0$ and the boundary data are nonnegative.

The inverse preserves every orthogonal symmetry shared by the domain,
weight, boundary data, nonlinearity and source.
\end{proposition}

\begin{proof}
Let $\mathcal G_\tau(X,t)=\int_0^t g_\tau(X,s)\dd s$.
Strong monotonicity and linear growth make
\[
 \mathcal J_{\tau,F}(u)=\int_D\mu_\tau
 \bigl(\tfrac12|\nabla u|^2+\mathcal G_\tau(X,u)-Fu\bigr)\dd X,
 \qquad u\in h+H_0^1(D),
\]
coercive, weakly lower semicontinuous and strictly convex. Its unique
minimizer satisfies
\begin{equation}\label{common:eq-T-weak-formulation}
 \int_D\mu_\tau\nabla u\cdot\nabla\varphi\dd X
 +\int_D\mu_\tau g_\tau(X,u)\varphi\dd X
 =\int_D\mu_\tau F\varphi\dd X, 
 \qquad\forall\ \varphi\in H_0^1(D).
\end{equation}

For comparison, let $v_1,v_2$ be a subsolution and a supersolution
with ordered sources and ordered boundary traces. Then
$w=(v_1-v_2)_+\in H_0^1(D)$, and subtraction gives
\[
 \int_D\mu_\tau|\nabla w|^2\dd X
 +\int_D\mu_\tau[g_\tau(X,v_1)-g_\tau(X,v_2)]w\dd X\le0.
\]
The second term is at least $\int_D\mu_\tau w^2$, so $w=0$.
This proves uniqueness and the range bound by comparison with $m,M$.
In particular, the constants $\pm B$, with
$B=\max\{\norm h_{C(\partial D)},\norm F_\infty\}$, are barriers
because $g_\tau(X,B)\ge B$ and $g_\tau(X,-B)\le-B$.
They give~\eqref{common:eq-T-bound}; zero is a subsolution for
nonnegative data. If $a=\norm{F-G}_\infty$ and
$v=\mathcal T_\tau G$, strong monotonicity makes $v+a$ a
supersolution for source $F$, with dominating boundary trace.
Interchanging $F,G$ proves~\eqref{common:eq-T-Lipschitz}.

For bounded source sets, comparison bounds $u$ and $g_\tau(X,u)$,
and minimality against $h$ bounds $\|u\|_{H^1(D)}$, uniformly in
$\tau$. Global Dirichlet estimates for
$-\Delta u=F-g_\tau(X,u)+\nabla\log\mu_\tau\cdot\nabla u$
then give uniform $W^{2,p}(D)$ bounds for every finite $p$ by
two-dimensional bootstrapping; see~\cite{GT}. The constants depend
only on $D,h$, the source bound and the uniform coefficient bounds. Then, the compact embedding $W^{2,p}(D) \hookrightarrow C(\overline{D})$, $p > 2$, gives collective compactness into $C(\overline D)$.

If $\tau_j\to\tau$ and $F_j\to F$ in $L^\infty(D)$, these bounds,
convergence of the weights and uniform convergence of $g_{\tau_j}$ on
bounded value intervals let us pass to~\eqref{common:eq-T-weak-formulation}.
Uniqueness identifies every subsequential limit with $\mathcal T_\tau F$,
proving joint continuity. The same uniqueness applied after an
orthogonal change of variables proves the symmetry assertion.
\end{proof}

\subsection{The singular linear inverse}

The first-order perturbation is expected to vanish on the reference
plateaus, so its equation is posed on the active annulus. Near its
boundary, the derivative of $g$ along the profile grows like
$H'''/H'\approx\log^2(1/d)/d^2$, by~\eqref{common:eq-Hratio3},
where $d$ is distance to that boundary. We construct the inverse of
$-\Delta+V$ for potentials $V$ with this growth. The natural solution
space requires both the Dirichlet energy and $\int Vv^2$ to be finite.
We first justify approximation in this space, then use comparison to
control the solutions near the singular boundary.

The following density result will be used both on annuli and
in the one-dimensional problems arising from the radial
spectral analysis. Let $\Omega$ be either a bounded open
interval in $\RR$ or a circular annulus in $\RR^2$, and set
$d=\operatorname{dist}(\cdot,\partial\Omega)$. Fix a smooth positive
weight $W\ge1$ in $\Omega$ such that
\begin{equation}\label{common:eq-weight}
 W\approx1+\frac{\log^2(1/d)}{d^2}
 \quad\hbox{near }\partial\Omega,
\end{equation}
and define
\begin{equation*}
 \mathcal E(\Omega,W)=H_0^1(\Omega)\cap L^2(\Omega,W\dd x),
 \qquad
 \norm v_{\mathcal E}^2=\int_\Omega(|\nabla v|^2+Wv^2)\dd x.
\end{equation*}
In the next lemma, whose proof is elementary, we gather some properties of this space. 

\begin{lemma}\label{common:lem-density}
The space $\mathcal E(\Omega,W)$ is Hilbert,
$C_c^\infty(\Omega)$ is dense in it, and its inclusion into
$L^2(\Omega)$ is compact. These conclusions remain valid when the
integrals are taken against a smooth density bounded above and below
by positive constants.
\end{lemma}

\begin{proof}
Completeness follows by taking the limits of a Cauchy sequence in
$H_0^1(\Omega)$ and $L^2(\Omega,W\dd x)$: they agree in $L^2(\Omega)$
since $W\ge1$. Compactness follows from Rellich's theorem.
For density, fix $v\in\mathcal E$ and choose smooth cutoffs
$\chi_\delta$ that vanish for $d\le\delta$, equal one for
$d\ge2\delta$, and satisfy $|\nabla\chi_\delta|\le C/\delta$.
Since $W\ge c/d^2$ near the boundary,
\[
 \|v-\chi_\delta v\|_{\mathcal E}^2
 \le C\int_{\{x\in\Omega:d(x)<2\delta\}}
          (|\nabla v|^2+W|v|^2)\dd x
 \longrightarrow0.
\]
For fixed $\delta$, mollification gives smooth compactly
supported approximants in the same norm, since $W$ is
bounded on compact subsets of $\Omega$.
A density bounded above and below by positive constants
gives equivalent norms, so the conclusions remain valid.
\end{proof}

To control the singular inverse at its boundary, we use the collar
coordinates of Subsection~\ref{common:sec-collars}. For
$\mathcal L=\Delta+b\cdot\nabla$, $0<2\delta<s_*$ and
$M\ge\max\{1,4B^2\}$, put $\kappa=\sqrt M/2$. The inequalities
\[
 (-\mathcal L+M)e^{\pm\kappa s}
 =(M-\kappa^2\mp\kappa\beta)e^{\pm\kappa s}
 \ge\tfrac M2e^{\pm\kappa s}
\]
show that, for $S,C_0\ge0$, the functions
\begin{align}
 W_2(s)&=\frac SM+C_0e^{-\kappa(s+2\delta)}
                       +C_0e^{-\kappa(2\delta-s)},\notag\\
 W_1(s)&=\frac SM+C_0e^{-\kappa(2\delta-s)},
 \label{common:eq-collar-barriers}
\end{align}
are supersolutions to $(-\mathcal L+M)W\ge S$ on the two-sided and
one-sided collars, respectively.

Fix $i\in\{1,2\}$ and let $\mathcal Q=\mathcal C_{2\delta}$
if $i=1$, and $\mathcal Q=\widetilde{\mathcal C}_{2\delta}$
if $i=2$. Suppose that
$\mathcal L=\mu^{-1}\Div(\mu\nabla)$, with $\mu$ smooth and
strictly positive on $\overline{\mathcal Q}$, and that
$v\in H^1(\mathcal Q)$ satisfies
\[
 -\mathcal Lv+cv=\eta,\qquad\text{in }\mathcal Q,
\]
where $\eta\in L^\infty(\mathcal Q)$, $|\eta|\le S$, and
$c\ge M$ wherever $v\ne0$. Setting $c=M$ on $\{v=0\}$,
we may assume $c\ge M$ throughout $\mathcal Q$.
Assume $cv\in L^2(\mathcal Q)$, or, more generally,
$\int_{\mathcal Q}c|v|^2\dd X<\infty$ with the equation
understood against all $\varphi\in H_0^1(\mathcal Q)$
such that $\int_{\mathcal Q}c|\varphi|^2\dd X<\infty$.

In the two-sided case, assume $|v|\le C_0$ on
$\Gamma_{-2\delta}\cup\Gamma_{2\delta}$.
In the one-sided case, assume $v=0$ on $\Gamma_0$ and
$|v|\le C_0$ on $\Gamma_{2\delta}$, all in the trace sense.
The boundary inequalities above give
$q=(v-W_i)_+\in H_0^1(\mathcal Q)$.
This test is admissible also in the finite-energy case,
since $0\le q\le v_+$.
On $\{q>0\}$, we have $cv-MW_i\ge Mq$.
Testing the difference of the equation and the supersolutions
therefore gives
\[
 \int_{\mathcal Q}\mu\bigl(|\nabla q|^2+Mq^2\bigr)\dd X\le0.
\]
Hence, it follows that $v\le W_i$. Applying the same argument to $-v$ yields
$|v|\le W_i$, and consequently
\begin{align} \label{common:eq-collar-damping-two}
 \|v\|_{L^\infty(\widetilde{\mathcal C}_\delta)}
 \le S/M+2C_0e^{-\delta\sqrt M/2}, \qquad
 \|v\|_{L^\infty(\mathcal C_\delta)}
 \le S/M+C_0e^{-\delta\sqrt M/2}.
 \notag
\end{align}
These comparisons give smallness near reference interfaces,
whereas Lemma~\ref{common:lem-exact-collar} gives exact
vanishing at the artificial boundary.
 
We now construct the inverse of the singular operator. We allow the
annulus to depend on a parameter, as it will in the planar construction.
Extending solutions by zero lets us compare these inverses in the same
space of continuous functions, despite the change in their form domains.
Let $J$ be a compact parameter interval, let $r_\pm:J\to(0,\infty)$ be continuous, and
suppose
\[
 \Omega_\tau=\{r_-(\tau)<|X|<r_+(\tau)\}\Subset D,
 \qquad \tau\in J,
\]
where $D\subset\RR^2$ is a fixed bounded smooth domain. Assume the
inner radii, the widths of the annuli, and their distances to
$\partial D$ have uniform positive lower bounds. Let
$V_\tau\ge1$ be smooth in $\Omega_\tau$. Assume that, for every
$\tau\in J$ and $K\Subset\Omega_\tau$, one has $K\Subset\Omega_\sigma$
for all $\sigma$ sufficiently close to $\tau$, and
$\norm{V_\sigma-V_\tau}_{C(K)}\to0$ as $\sigma\to\tau$.
Assume also that the potentials are uniformly comparable to weights
$W_\tau$ as in~\eqref{common:eq-weight}. More precisely, the endpoint comparison
holds on collars of a common width with common constants, and on
the complement the potentials are uniformly bounded.
Set
\[
 \mathcal E_\tau=H_0^1(\Omega_\tau)\cap
                    L^2(\Omega_\tau,V_\tau\dd X),
 \qquad
 \mathfrak a_\tau(v,w)=\int_{\Omega_\tau}
           (\nabla v\cdot\nabla w+V_\tau vw)\dd X.
\]
We equip $\mathcal E_\tau$ with the natural norm
$\norm v_{\mathcal E_\tau}^2=\mathfrak a_\tau(v,v)$.

\begin{proposition}\label{common:prop-singular-resolvent}
For every $\tau\in J$ and $F\in L^2(D)$, there is a unique
$\mathcal R_\tau F\in\mathcal E_\tau$ such that
\begin{equation*}
 \mathfrak a_\tau(\mathcal R_\tau F,\eta)
 =\int_{\Omega_\tau}F\eta\dd X,
 \qquad \eta\in\mathcal E_\tau.
\end{equation*}
It satisfies
\begin{align}
 \norm{\mathcal R_\tau F}_{\mathcal E_\tau}
   &\le\norm F_{L^2(\Omega_\tau)},\label{common:eq-R-energy}\\
 \norm{\mathcal R_\tau F}_{L^\infty(\Omega_\tau)}&\le\norm F_{L^\infty(D)},
       \quad F\in L^\infty(D).\label{common:eq-R-bound}
\end{align}
Extend $\mathcal R_\tau F$ by zero outside $\Omega_\tau$. Then
$\mathcal R_\tau:L^\infty(D)\to C_0(\overline D)$ is compact, the family $\{\mathcal R_\tau\}_{\tau\in J}$ is collectively
compact from bounded subsets of $L^\infty(D)$ into
$C_0(\overline D)$, and
\begin{equation}\label{common:eq-R-continuity}
 \norm{\mathcal R_\tau-\mathcal R_\sigma}_{L^\infty(D)\to C(\overline D)}
 \longrightarrow0,\qquad \tau\to\sigma.
\end{equation}
If, in addition, $\tau_j\to\tau$ and $F_j\rightharpoonup^{*}F$ in a bounded subset
of $L^\infty(D)$, then
\begin{equation}\label{common:eq-R-weakstar}
 \mathcal R_{\tau_j}F_j\longrightarrow \mathcal R_\tau F
 \quad\hbox{in }C(\overline D).
\end{equation}
Finally, if $Q$ is orthogonal, $Q\Omega_\tau=\Omega_\tau$, and
$V_\tau\circ Q=V_\tau$, the zero extension of $\mathcal R_\tau F$
to $\RR^2$ is $Q$-invariant whenever $F|_{\Omega_\tau}$ is
$Q$-invariant.
\end{proposition}
\begin{proof}
We first prove existence, uniqueness and the estimates
\eqref{common:eq-R-energy} and \eqref{common:eq-R-bound}.
Lemma~\ref{common:lem-density} makes $\mathcal E_\tau$ a Hilbert space
with inner product $\mathfrak a_\tau$. Since $V_\tau\ge1$, the source
functional has norm at most $\norm F_{L^2(\Omega_\tau)}$. The Riesz representation theorem gives a unique $v=\mathcal R_\tau F$ and~\eqref{common:eq-R-energy}. 

Next, for $F\in L^\infty(D)$, set $A=\|F\|_{L^\infty(D)}$.
The truncations $(v-A)_+$ and $(-v-A)_+$ belong to
$\mathcal E_\tau$, since they have zero trace and are
bounded by $|v|$. Testing the weak equation with these
truncations, as in the proof of
Proposition~\ref{common:prop-monotone-inverse}, gives
\eqref{common:eq-R-bound}. Interior elliptic regularity, see \cite{GT}, gives $v\in W_{\mathrm{loc}}^{2,p}(\Omega_\tau)$ for every
$1<p<\infty$, and hence a continuous representative
in $\Omega_\tau$, still denoted by $v$.

We now handle the  boundary behavior. We extend $\mathcal{R}_\tau F$ by zero outside $\Omega_\tau$. First, by linearity, it suffices to consider
$\|F\|_{L^\infty(D)}\le1$. Also, we introduce the notation $d_\tau=\dist(\cdot,\partial\Omega_\tau)$. The uniform lower comparison for $V_\tau$ gives, for all
sufficiently small $\delta>0$, independently of $\tau$, 
\[
 V_\tau\ge M_\delta:=c\,\frac{\log^2(1/\delta)}{\delta^2}
 \quad\hbox{on }0<d_\tau<2\delta.
\] 
The inner-radius bound controls the radial coefficient of $\Delta$.
Thus, in either collar with $s=d_\tau$, the one-sided supersolution
from~\eqref{common:eq-collar-barriers} is
\[
 B_\delta(s)=M_\delta^{-1}
       +e^{-\sqrt{M_\delta}(2\delta-s)/2},\qquad
 (-\Delta+M_\delta)B_\delta\ge1.
\]
At $s=0$, the trace of $v$ is zero and $B_\delta(0)>0$.
At $s=2\delta$, we have $|v|\le1<B_\delta(2\delta)$.
Applying the preceding comparison to $v$ and $-v$, with
potential $c = V_\tau$ and $M=M_\delta$, therefore gives
\[
 \sup_{\{d_\tau\le\delta\}}|\mathcal R_\tau F|
 \le\eta(\delta):=M_\delta^{-1}
                    +e^{-\delta\sqrt{M_\delta}/2}\longrightarrow0.
\]
The convergence is thus uniform in $\tau$ and the unit source ball, and
makes the zero extension continuous on $\overline D$. 

Concerning the compactness and the parameter dependence, take $\tau_j\to\tau$ and sources with $\norm{F_j}_{L^\infty(D)}\le S$.
Since $L^1(D)$ is separable, a subsequence satisfies
$F_j\rightharpoonup^*F$ in $L^\infty(D)$. The zero extensions
$v_j=\mathcal R_{\tau_j}F_j$ are bounded in $H_0^1(D)$ by
\eqref{common:eq-R-energy}. On each compact subset of $\Omega_\tau$,
the nearby annuli contain a common neighborhood and the potentials
are uniformly bounded. Interior $W^{2,p}$ estimates, see e.g. \cite{GT}, therefore give subsequential uniform convergence there.

For $K_\delta=\{X\in\Omega_\tau:d_\tau(X)\ge2\delta\}$, once the
boundary radii differ by less than $\delta$, every point of
$\Omega_{\tau_j}\setminus K_\delta$ has $d_{\tau_j}<3\delta$. Hence
\begin{align*}
 \sup_{\overline D\setminus K_\delta}|v_j|&\le S\eta(3\delta),\\
 \norm{v_j-v_k}_{C(\overline{D})}
 &\le\norm{v_j-v_k}_{C(K_\delta)}+2S\eta(3\delta)
\end{align*}
for large $j,k$. A diagonal subsequence converging on all these
interior sets is consequently Cauchy on $\overline D$, and thus we have the desired compactness.

Next, we identify the limit to prove the continuity statements.
Denote the uniform limit by $v$. After passing to a further
subsequence, we also have $v_j\rightharpoonup v$ in $H_0^1(D)$. Outside $\overline\Omega_\tau$ the sequence eventually
vanishes; continuity and the trace characterization on the smooth
annulus give $v|_{\Omega_\tau}\in H_0^1(\Omega_\tau)$. Moreover,
for $K\Subset\Omega_\tau$,  
\[
 \int_KV_\tau v^2\dd X
 =\lim_j\int_KV_{\tau_j}v_j^2\dd X\le |D| S^2.
\]
Exhaustion gives $v\in\mathcal E_\tau$. For every $\eta \in C_c^{\infty}(\Omega_\tau)$, weak gradient convergence, local uniform convergence
of the potential term and weak-star convergence of the sources give
\[
 \mathfrak a_\tau(v,\eta)=\int_{\Omega_\tau}F\eta\dd X.
\]
Lemma~\ref{common:lem-density} extends this identity to every $\eta \in \mathcal{E}_\tau$, and uniqueness gives $v=\mathcal R_\tau F$.

Since $J$ is compact, this argument proves collective compactness.
When the sources already converge weak-star, every uniform
subsequential limit is the same, proving~\eqref{common:eq-R-weakstar}.
If~\eqref{common:eq-R-continuity} failed, there would be
$\tau_j\to\tau$ and $\norm{F_j}_{L^\infty(D)}\le1$ with
$\norm{\mathcal R_{\tau_j}F_j-\mathcal R_\tau F_j}_{C(\overline{D})}\ge\epsilon>0$.
Extracting a weak-star convergent source subsequence and applying
\eqref{common:eq-R-weakstar} at varying and fixed parameters makes
both terms converge to $\mathcal R_\tau F$, a contradiction.

Finally, an orthogonal change of variables preserving the data
preserves the form equation, so uniqueness proves symmetry.
\end{proof}

\subsection{A common linearization argument}

Subtracting two nonlinear inverse equations and dividing by the size
of the perturbation gives a linear equation whose potential is a
secant slope of $g$. Inside the active annulus, this slope converges
to the potential of the singular inverse. Near the reference
interfaces it becomes large and forces the quotient to be small.
On the plateaus, the source must also be small relative to the slope.
The following criterion combines these controls to obtain uniform
convergence and identify the first-order response in both constructions.
In both applications the exterior source is actually zero; condition~(ii)
also allows a residual that is small relative to the secant slope.

\begin{proposition}\label{common:prop-linearization}
Under the assumptions of
Proposition~\ref{common:prop-singular-resolvent}, let
$\tau_j\to\tau$ in $J$, let $\mu_j$ be positive smooth functions on
$\overline D$ such that $\mu_j\to1$ in $C^1(\overline D)$, and set
\[
 L_jv=\mu_j^{-1}\Div(\mu_j\nabla v).
\]
Let $c_j:D\to[1,\infty)$ be measurable, and let $(\eta_j)$ be
bounded in $L^\infty(D)$. Suppose that
$z_j\in H_0^1(D)\cap C(\overline D)$ satisfies
$c_jz_j\in L^2(D)$ and solves
\begin{equation}\label{common:eq-quotient}
 -L_jz_j+c_jz_j=\eta_j,\qquad\text{in }D,
\end{equation}
in the weak sense. Assume that $c_j\to V_\tau$ uniformly on compact subsets of
$\Omega_\tau$, and that the following bounds hold almost
everywhere on $\{X\in D:z_j(X)\ne0\}$:
\begin{enumerate}
 \item There exist constants $c_*>0$ and $\delta_0>0$ such that,
 for every $0<\delta<\delta_0$ and all sufficiently large $j$,
 depending on $\delta$,
 \[
  c_j(X)\ge M_\delta:=
  c_*\frac{\log^2(1/\delta)}{\delta^2},
  \qquad \dist(X,\partial\Omega_{\tau_j})<2\delta.
 \]
 \item There is a nonnegative sequence $(\sigma_j)$ with $\sigma_j\to0$
 such that
 \[
  |\eta_j(X)|\le\sigma_jc_j(X),
  \qquad X\in D\setminus\Omega_{\tau_j}.
 \]
\end{enumerate}
Then
\begin{equation}\label{common:eq-linearization-limit}
 \norm{z_j-\mathcal R_{\tau_j}\eta_j}_{C(\overline D)}
 \longrightarrow0.
\end{equation}
In particular, if $\eta_j\rightharpoonup^*\eta$ in $L^\infty(D)$,
then $z_j\to\mathcal R_\tau\eta$ in $C(\overline D)$.
\end{proposition}
 
\begin{proof}
We first establish uniform bounds. Set
$
 S:=\sup_j\|\eta_j\|_{L^\infty(D)}.
$
The truncation argument in
Proposition~\ref{common:prop-singular-resolvent}, followed by
testing with $z_j$, gives
\begin{equation}\label{common:eq-quotient-bounds}
 \|z_j\|_{L^\infty(D)}\le S,\qquad
 \int_D\mu_j\bigl(|\nabla z_j|^2+c_jz_j^2\bigr)\dd X
 \le CS^2.
\end{equation}
Since the weights $\mu_j$ have uniform
positive lower and upper bounds and $c_j\ge1$, the sequence
$(z_j)$ is bounded in $H_0^1(D)$.

We next control $z_j$ near $\partial\Omega_{\tau_j}$. We first write
\[
 U_{j,\delta}:=
 \{X\in D:\dist(X,\partial\Omega_{\tau_j})<\delta\},
\]
and choose $\delta>0$ sufficiently small so that $U_{j,2\delta}$
consists of two disjoint collars contained in $D$, uniformly
in $j$. The vector fields
$b_j:=\nabla\log\mu_j$ are uniformly bounded in
$L^\infty(D;\RR^2)$. Together with the lower bound for the
inner radii, this gives a uniform bound for the radial
first-order coefficients of $L_j$ on these collars.
As $M_\delta\to\infty$ when $\delta\downarrow0$, we may also
assume that $M_\delta$ satisfies the size condition required
for the comparison in~\eqref{common:eq-collar-barriers}.

For all sufficiently large $j$, compare $z_j$ and $-z_j$
with the two-sided barrier $W_2$, taking $M=M_\delta$ and
$C_0=S$. The boundary traces are bounded by $S$.
Moreover, the tests $(\pm z_j-W_2)_+$ vanish wherever
$z_j=0$, so assumption~(i) provides the required lower
bound for $c_j$ on their supports. The comparison therefore
gives
\begin{equation}\label{common:eq-quotient-collar}
 \|z_j\|_{L^\infty(U_{j,\delta})}
 \le\omega(\delta):=
 \frac{S}{M_\delta}
 +2S e^{-\delta\sqrt{M_\delta}/2},
 \qquad \omega(\delta)\longrightarrow0
 \quad\text{as }\delta\downarrow0.
\end{equation}

To control the remaining part of the complement, set
\[
 E_j:=D\setminus\overline{\Omega}_{\tau_j},
 \qquad a_j:=\max\{\omega(\delta),\sigma_j\}.
\]
The function $q_j:=(z_j-a_j)_+|_{E_j}$ belongs to
$H_0^1(E_j)$. Indeed, its trace vanishes on $\partial\Omega_{\tau_j}$
by~\eqref{common:eq-quotient-collar}, and on $\partial D$
because $z_j\in H_0^1(D)$. Its zero extension is therefore
an admissible test in $H_0^1(D)$.
On $\{X\in E_j:q_j(X)>0\}$, assumption~(ii) gives
\[
 \eta_j\le\sigma_jc_j\le a_jc_j.
\]
Testing the equation with this extension yields
\[
 \int_{E_j}\mu_j|\nabla q_j|^2\dd X
 +\int_{E_j}\mu_jc_jq_j^2\dd X\le0.
\]
The potential term is finite because $q_j\le|z_j|$ and
$c_jz_j\in L^2(D)$. Hence $q_j=0$. Applying the same
argument to $-z_j$ and using continuity on the annular
boundary, we obtain
\[
 \sup_{X\in D\setminus\Omega_{\tau_j}}|z_j(X)|
 \le\max\{\omega(\delta),\sigma_j\}
\]
for every sufficiently small $\delta>0$ and all sufficiently
large $j$.

We now proceed as in the proof of Proposition \ref{common:prop-singular-resolvent}. Interior elliptic estimates give convergence on compact subsets of the annulus. The estimates near its boundary and in its complement extend this convergence to the whole domain. We then identify the limit using the energy bound and the weak formulation. For an arbitrary subsequence, extract
$\eta_j\rightharpoonup^*\eta$ in $L^\infty(D)$ and
$z_j\rightharpoonup z$ in $H_0^1(D)$. Here $\mu_j\to1$ in $C^1$,
so the drifts tend to zero, and the secant slopes $c_j$ converge locally
uniformly to $V_\tau$. Thus interior estimates again give, after a further
extraction, local uniform convergence in $\Omega_\tau$. The collar and exterior bounds just proved
replace the boundary damping and exact zero extensions in that
proposition. They give convergence in $C(\overline D)$ and a limit
vanishing outside $\Omega_\tau$, with
$z|_{\Omega_\tau}\in H_0^1(\Omega_\tau)$. Indeed, set
\[
 K_\delta:=\{X\in\Omega_\tau:
       \dist(X,\partial\Omega_\tau)\ge2\delta\}.
\]
For every sufficiently small fixed $\delta>0$, once the boundary
radii differ by less than $\delta$, each point of
$\Omega_{\tau_j}\setminus K_\delta$ lies in $U_{j,3\delta}$.
The collar and exterior estimates therefore give
\[
 \sup_{X\in\overline D\setminus K_\delta}|z_j(X)|
 \le\max\{\omega(3\delta),\sigma_j\}
\]
for all sufficiently large $j$, with a threshold depending on $\delta$.
First letting $j\to\infty$ and then $\delta\downarrow0$, this estimate
and uniform convergence on $K_\delta$ give convergence on $\overline D$.

The new energy bound~\eqref{common:eq-quotient-bounds} supplies the
remaining singular-domain condition: for every $K\Subset\Omega_\tau$, it follows that
\[
 \int_K V_\tau z^2\dd X
 =\lim_j\int_K\mu_jc_jz_j^2\dd X\le CS^2.
\]
Exhaustion gives $z\in\mathcal E_\tau$. Convergence of the weights,
local convergence of the slopes, weak convergence of the gradients and
weak-star convergence of the sources now give
$\mathfrak a_\tau(z,\xi)=\int_{\Omega_\tau}\eta\xi\dd X$
first for $\xi\in C_c^\infty(\Omega_\tau)$ and then, by
Lemma~\ref{common:lem-density}, for all $\xi\in\mathcal E_\tau$.
Uniqueness identifies $z=\mathcal R_\tau\eta$.
By~\eqref{common:eq-R-weakstar}, $\mathcal R_{\tau_j}\eta_j$ has
the same uniform limit. Since the starting subsequence was arbitrary,
this proves~\eqref{common:eq-linearization-limit}; the final assertion
follows from~\eqref{common:eq-R-weakstar}.
\end{proof}

\section{Nonlocalizable three-dimensional Euler flows}
\label{s:sec:profile}

We construct the three-dimensional flows by placing the radial reference profile at a large distance \(R\) from the symmetry axis and solving the resulting Grad–Shafranov equation as a perturbation of the planar equation. The swirl parameter is chosen to satisfy the compatibility condition associated with the translation mode. Sections \ref{common:sec-prelim} and \ref{common:sec-compact} provide smooth localization, while the first-order correction shows that the pressure varies along some streamlines.

\begin{theorem}\label{thm:spatial-precise}
There exist $a>1$, $\delta\in(0,1/8)$ and $R_*>a+2$ such that, for every
$R>R_*$, equation~\eqref{eq:euler} in dimension three admits a solution
\[
 u_R\in C_c^\infty(\RR^3;\RR^3),\qquad
 p_R\in C_c^\infty(\RR^3).
\]
The velocity is nonzero, axisymmetric with nonzero swirl, and satisfies
\[
 u_R\cdot\nabla p_R\not\equiv0.
\]
In cylindrical coordinates $(r,\varphi,z)$, its support obeys
\begin{equation}\label{eq:spatial-support}
 \supp u_R\Subset
 \left\{(r\cos\varphi,r\sin\varphi,z):
 a-\delta<\sqrt{(r-R)^2+z^2}<a+1+\delta\right\}.
\end{equation}
More precisely, there are numbers $\lambda_R>0$ and smooth functions
$\psi_R:\RR^2\to[0,2]$ such that $u_R$ is given by
\eqref{s:eq:re-velocity} with
$\Psi_R(r,z)=\psi_R(r-R,z)$ and
$F(t)=\sqrt{2\lambda_R}\,t(2-t)$.
The velocity extends smoothly by zero to the symmetry axis.
In particular, there is a smooth profile $\phi$, equal to $0$ on $[0,a]$, strictly
increasing on $(a,a+1)$ and equal to $2$ on $[a+1,\infty)$, for which
\begin{equation*}
 \psi_R\longrightarrow U:=\phi\bigl(|\cdot|)
 \quad\textup{in }C^q(\RR^2)\textup{ for every integer }q\geq0,
 \qquad \lambda_R\longrightarrow\lambda_*>0.
\end{equation*}
Both limits are taken as $R\to\infty$. Moreover,
$\norm{\psi_R-U}_{C(\RR^2)}=O(R^{-1})$ as $R\to\infty$, and all the functions
$\psi_R-2$ have support in one fixed compact set.
\end{theorem}

\subsection{Scalar reduction and reference profile}\label{s:sec:scalar-reduction}

We first choose the scalar equation and its radial reference. In
cylindrical coordinates $(r,\varphi,z)$, a stream function $\Psi$, a swirl function $F$, and a Bernoulli function $B$ define
\begin{equation}\label{s:eq:re-velocity}
 u=\frac1r(-\Psi_z e_r+F(\Psi)e_\varphi+\Psi_r e_z), \qquad
 p=B(\Psi)-\tfrac12|u|^2.
\end{equation}
This velocity is divergence free. Writing
$\Delta_*\Psi=\Psi_{rr}+\Psi_{zz}-r^{-1}\Psi_r$, direct differentiation
gives
\[
 u\times\curl u=\frac{\Delta_*\Psi+F(\Psi)F'(\Psi)}{r^2}\nabla\Psi.
\]
Thus, Euler follows from the Grad--Shafranov equation
\begin{equation}\label{s:eq:GS-physical}
 \Delta_*\Psi=r^2B'(\Psi)-F(\Psi)F'(\Psi).
\end{equation}
This classical reduction is used in~\cite{Gavrilov,CLV,DEPS}. We will
choose $F$ to vanish at the plateau values $0$ and $2$, so constant
regions of $\Psi$ give stagnant regions of the flow. The later
reconstruction checks smoothness at those regions and at the axis.

There are two distinct radii. The radius $a$ describes the annulus in
the centered meridian plane and will be fixed sufficiently large for
spectral separation. The torus radius $R$ then tends to infinity with
$a$ fixed. In the coordinates $x=r-R$, $z=z$, put
$\varepsilon=R^{-1}$. The operator $\Delta_*$ then becomes
\[
 L_\varepsilon=\Delta-\frac{\varepsilon}{1+\varepsilon x}\partial_x,
\]
so toroidal curvature is a small perturbation of the planar reference
equation. We now build that reference state and choose the scalar data.

The reference stream function must join the two plateaus monotonically
and have a value law with the nearly linear endpoint bounds. We use
$H$ from~\eqref{common:eq-H} near each endpoint and a monotone
interpolation in between.

Choose $\chi\in C^\infty(\RR)$ nondecreasing, equal to zero for $s\le1/3$
and to one for $s\ge2/3$, and define
\begin{equation*}
 \Phi(s)=(1-\chi(s))H(s)+\chi(s)\bigl(2-H(1-s)\bigr), \qquad 0 < s < 1.
\end{equation*}
Both interpolated functions are strictly increasing, the second is
larger than the first, and $\chi$ is nondecreasing. Thus $\Phi'>0$.
Extend $\Phi$ by $0$ for $s\le0$ and by $2$ for $s\ge1$. This gives a
single $C^\infty(\RR)$ function, with infinite-order matching at $s = 0$ and $s = 1$, with plateau values $0$ and $2$.

For $a\ge2$, we then set
\begin{equation*}
 \phi_a(\rho)=\Phi(\rho-a),\qquad
 v_a(\rho)=\phi_a'(\rho),\qquad
 \Omega_a=\{(x,z):a<\sqrt{x^2+z^2}<a+1\}.
\end{equation*}
The notation $\phi_a$ includes its two constant extensions.
Since $\Phi$ is invertible on $(0,1)$, we can define $f_a:(0,2)\to\RR$ by
\begin{equation}\label{s:eq:f-definition}
 f_a(\phi_a(\rho))=\phi_a''(\rho)+\frac{\phi_a'(\rho)}\rho.
\end{equation}
This is smooth on $(0,2)$ and extends continuously to $[0,2]$ by
$f_a(0)=f_a(2)=0$. Moreover, it follows that
\begin{equation} \label{eq:inner-endpoint-formulas}
\begin{aligned}
 f_a(H(s))&=H''(s)+\frac{H'(s)}{a+s}, \qquad 0 < s \ll 1,\\
 f_a(2-H(s))&=-H''(s)+\frac{H'(s)}{a+1-s}, \qquad 0 < s \ll 1.
\end{aligned}
\end{equation}
Also, setting $U_a(x,z)=\phi_a(\sqrt{x^2+z^2})$, with $x=\rho\cos\theta$ and $z=\rho\sin\theta$, we have that $\Delta U_a = f_a (U_a)$ in $\RR^2$. Note as well that the gradient of $U_a$ is compactly supported. 

For each $a\ge2$, choose the swirl
\[
 F_\lambda(t)=\sqrt{2\lambda}\,t(2-t),\qquad
 q(t)=t^2(2-t)^2,\qquad \lambda>0.
\]
Thus $F_\lambda^2/2=\lambda q$, and the swirl vanishes at both
plateaus. We choose the Bernoulli derivative so that the limiting
equation remains $\Delta\psi=f_a(\psi)$ for every $\lambda$, i.e. we choose
\[
 h_{a,\lambda}=f_a+\lambda q',\qquad
 B'_{a,\varepsilon,\lambda}=\varepsilon^2h_{a,\lambda},
 \qquad B_{a,\varepsilon,\lambda}(2)=0.
\]
The reference state is therefore fixed while $\lambda$ remains
available to cancel the curvature forcing in the translation mode.
Fix $\delta\in(0,1/8)$ and enlarge the reference annulus to
\[
 D_a=\{a-\delta<\rho<a+1+\delta\}.
\]
For $\varepsilon$ small enough that $1+\varepsilon x\ge1/2$ near
$\overline D_a$, set
\begin{align*}
 A_\varepsilon(x)&=(1+\varepsilon x)^2,\\
 N_{a,\varepsilon,\lambda}(x,t)
 &=A_\varepsilon(x)f_a(t)+(A_\varepsilon(x)-1)\lambda q'(t)
  =A_\varepsilon(x)h_{a,\lambda}(t)-\lambda q'(t).
\end{align*}
Since $r^2\varepsilon^2=A_\varepsilon$ and
$F_\lambda F_\lambda'=\lambda q'$, equation~\eqref{s:eq:GS-physical}
becomes
\begin{equation}\label{s:eq:GS-centered}
 -L_\varepsilon\psi+N_{a,\varepsilon,\lambda}(x,\psi)=0,
 \qquad \psi|_{\partial D_a}=U_a|_{\partial D_a}.
\end{equation}
The endpoint formulas above also provide the bounds needed for spectral
analysis and regularity.
Differentiating the inner endpoint formulas \eqref{eq:inner-endpoint-formulas} with respect to the profile value gives
\begin{equation}\label{s:eq:fprime-inner}
\begin{aligned}
 f_a'(H(s))
& =\frac{H'''(s)}{H'(s)}+\frac{H''(s)}{(a+s)H'(s)}-\frac1{(a+s)^2}, \qquad 0< s \ll 1, \\
 f_a'(2-H(s)) & = \frac{H'''(s)}{H'(s)}-\frac{H''(s)}{(a+1-s)H'(s)}-\frac1{(a+1-s)^2}, \qquad 0 < s \ll 1.
\end{aligned}
\end{equation}
By~\eqref{common:eq-Hratio2}--\eqref{common:eq-Hratio3}, the leading term
has order $\log^2(1/s)/s^2$, whereas the curvature corrections have orders
$\log(1/s)/s$ and $1$. Hence, uniformly for $a\ge2$,
\begin{equation}\label{s:eq:endpoint-growth}
 f_a'(H(s))\approx f_a'(2-H(s))\approx\frac{\log^2(1/s)}{s^2}.
\end{equation}
The value estimates~\eqref{common:eq-value-calculus}, applied to the two
endpoint formulas, also give, for every integer $j\ge0$ and every $\eta>0$,
\begin{equation}\label{s:eq:nearlinear-f}
 |f_a^{(j)}(t)|+|f_a^{(j)}(2-t)|
 \le C_{j,\eta}t^{1-j-\eta},\qquad 0<t<t_c,
\end{equation}
with constants uniform for $a\ge2$, since the curvature coefficients
and their derivatives are uniformly bounded.

\subsection{The translational kernel}

We next choose $a$ for its spectral properties, and only then let
$\varepsilon\downarrow0$ with $a$ fixed. Translations give two elements of the linearized kernel. This subsection is devoted to proving that, actually, they are the only ones for $a$ sufficiently large. 

First, we write $y(s)=\Phi'(s)>0$ for $0<s<1$. Differentiating
\eqref{s:eq:f-definition} gives
\begin{equation}\label{s:eq:spec-potential-exact}
 P_a(s):=f_a'(\Phi(s))
 =\frac{y''(s)}{y(s)}+
   \frac{y'(s)}{(a+s)y(s)}-\frac1{(a+s)^2}.
\end{equation}
Then, we choose a smooth weight $W\ge1$ on $(0,1)$ with the endpoint growth
in~\eqref{common:eq-weight}. The calculation leading to
\eqref{s:eq:endpoint-growth} gives, uniformly for $a\ge2$,
\begin{equation}\label{s:eq:spec-potential-bounds-a}
 cW-C\le P_a\le CW,\qquad |y'/y|\le C\sqrt W.
\end{equation}
Indeed, near either endpoint $y''/y$ is bounded above and below by positive multiples of $W$, while $|y'/y| \leq C \sqrt{W}$. Since $a \geq 2$, the remaining terms in \eqref{s:eq:spec-potential-exact} have absolute value at most $C \sqrt{W} + C$, giving the estimates for $P_a$ near the endpoints. On a fixed middle interval $[\eta, 1-\eta]$, the functions $y'/y, y''/y$, and $W$ are bounded because $y$ is smooth and strictly positive. Enlarging the constants therefore gives \eqref{s:eq:spec-potential-bounds-a}.
 
Next, using the energy space of Section~\ref{common:sec-compact}, set
\begin{equation}\label{s:eq:spec-planar-domain}
 \mathcal E_a=\mathcal E(\Omega_a,W(\rho-a)),\qquad
 \mathcal Q_a[h]=\int_{\Omega_a}
       (|\nabla h|^2+P_a(\rho-a)|h|^2)\dd x\dd z.
\end{equation}
The potential bounds and Lemma~\ref{common:lem-density} make this a
closed, lower bounded form with compact resolvent. It defines the
finite-energy realization of $\mathcal L_a=-\Delta+f_a'(U_a)$.

\begin{proposition}\label{s:prop:spec-planar-kernel}
There exists $a_{\mathrm s}\ge2$ such that, for every $a\ge a_{\mathrm s}$,
\begin{equation*}
 \ker\mathcal L_a
 =\operatorname{span}\{v_a(\rho)\cos\theta,\ v_a(\rho)\sin\theta\}.
\end{equation*}
On functions even under $(x,z)\mapsto(x,-z)$, it reduces to
\begin{equation*}
 \operatorname{span}\{Z\},\qquad
 Z(x,z):=v_a(\rho)\cos\theta=\partial_x\phi_a(\rho).
\end{equation*}
Furthermore, $\mathcal L_a$ has exactly one negative eigenvalue,
counted with multiplicity, and its eigenfunction is radial. On even
functions orthogonal to $Z$, it has a bounded inverse from
$L^2(\Omega_a)$ into its operator domain, and hence into $\mathcal E_a$.
\end{proposition}

\begin{proof}
We use, with a slight abuse of notation, $\mathcal E_1=\mathcal E((0,1),W)$, with squared norm
$\int_0^1(|h'|^2+W|h|^2)\dd s$, and stress that, for fixed $a$, the measures
$\dd s$ and $(a+s)\dd s$ give equivalent norms.
Then, for $n\ge0$, we consider the forms
\begin{equation}\label{s:eq:spec-mode-form}
 \mathfrak q_{a,n}[h]
 :=\int_0^1(a+s)\left(
 |h'|^2+\left[P_a(s)+\frac{n^2}{(a+s)^2}\right]|h|^2
 \right)\dd s,\qquad h\in\mathcal E_1,
\end{equation}
and denote by $\mu_1(a,n)\le\mu_2(a,n)\le\cdots$ the eigenvalues
of this form in $L^2((0,1),(a+s)\dd s)$, counted with multiplicity.
The energy domain fixes the realization at the singular endpoints.

\step{The ground state.}
For $h\in C_c^\infty(0,1)$, substitute $h=y\xi$ into
\eqref{s:eq:spec-mode-form} and integrate the cross term.
Formula~\eqref{s:eq:spec-potential-exact} cancels the zeroth-order terms,
giving
\begin{equation*}
 \mathfrak q_{a,1}[h]
 =\int_0^1(a+s)y(s)^2
       \left|\left(\frac{h}{y}\right)'\right|^2\dd s.
\end{equation*}
The identity extends to $\mathcal E_1$ by
Lemma~\ref{common:lem-density}: form convergence makes the weighted
derivatives on the right Cauchy, and local $H^1$ convergence identifies
their limit.

Since $y(s)=H'(s)$ near $0$ and $y(s)=H'(1-s)$ near $1$,
flatness gives $y\in\mathcal E_1$.
The previous identity and polarization show that $y$ lies in the operator kernel,
and equality forces $h/y$ to be constant. Thus, zero is a simple
first eigenvalue and $\mu_2(a,1)>0$.

\step{A uniform spectral gap.}
The ground-state identity compares these forms directly with the
one at $a=2$. Put $w_a(s)=1+s/a$. For every $a\ge2$ and $0<s<1$,
\[
 \tfrac23 w_2(s)\le w_a(s)\le w_2(s).
\]
Consequently, for every $h\in\mathcal E_1$,
\[
 \frac{\mathfrak q_{a,1}[h]}{a}
 \ge\frac23\frac{\mathfrak q_{2,1}[h]}2,
 \qquad
 \frac1a\int_0^1(a+s)|h|^2\dd s
 \le\frac12\int_0^1(2+s)|h|^2\dd s.
\]
The forms are nonnegative and have a common domain, so the min--max
principle gives
\[
 \mu_2(a,1)\ge\tfrac23\mu_2(2,1)>0.
\]

\step{The remaining angular modes.}
\begin{equation*}
 \mathfrak q_{a,n}[h]=\mathfrak q_{a,1}[h]
 +(n^2-1)\int_0^1\frac{|h|^2}{a+s}\dd s.
\end{equation*}
Testing with $y$ gives $\mu_1(a,0)<0$, while min--max gives
$\mu_2(a,0)\ge\mu_2(a,1)-a^{-2}>0$ for large $a$.
For $n\ge2$, nonnegativity of $\mathfrak q_{a,1}$ and
$(a+s)^{-2}\ge(a+1)^{-2}$ give, for every $a\ge2$,
\[
 \mu_1(a,n)\ge\frac{n^2-1}{(a+1)^2}>0.
\]

\step{The two-dimensional operator.}
The angular decomposition of the closed form $\mathcal Q_a$ is the
orthogonal sum of the forms $\mathfrak q_{a,n}$, with cosine and sine
sectors for $n\ge1$.
The preceding mode bounds give the claimed kernel and the unique
negative radial eigenvalue for all sufficiently large $a$. Reflection
in $z$ retains the first cosine mode and removes the sine mode.

Since zero is isolated, $\mathcal L_a$ has a bounded inverse from the even
functions in $L^2(\Omega_a)$ orthogonal to $Z$ into the operator domain,
and hence into $\mathcal E_a$.
\end{proof}

\subsection{Curvature defect and compatibility}

The curvature forcing must be orthogonal to the remaining translation
mode. We compute its projection to determine the swirl parameter. From now on, 
we fix $a\ge a_{\mathrm s}$ as in
Proposition~\ref{s:prop:spec-planar-kernel} and suppress the subscript
$a$ in $\phi_a,f_a,U_a,\Omega_a,D_a,h_{a,\lambda}$ and
$N_{a,\varepsilon,\lambda}$. In particular, we write
\[
 v=\phi',\qquad Z=U_x=v(\rho)\cos\theta,\qquad
 \mathcal L=-\Delta+f'(U).
\]

To determine the correction, we first measure how far $U$ is from
solving the curved equation. Its residual is exactly
\begin{align}
 E_{\varepsilon,\lambda}
 :=-L_\varepsilon U+N_{\varepsilon,\lambda}(x,U) =\frac{\varepsilon}{1+\varepsilon x}U_x+(2\varepsilon x+\varepsilon^2x^2)h_\lambda(U),
 \label{s:eq:reference-defect}
\end{align}
so, uniformly in $C(\overline D)$ for bounded $\lambda$,
\begin{equation*}
 \frac{E_{\varepsilon,\lambda}}\varepsilon\longrightarrow E_{1,\lambda},
 \qquad E_{1,\lambda}=U_x+2x h_\lambda(U).
\end{equation*}
All these defects vanish on both reference plateaus $D\setminus\overline\Omega$.
The first-order equation for a correction $\psi=U+\varepsilon w$ is then
\[
 \mathcal Lw=-E_{1,\lambda}.
\]
Since the even kernel is spanned by $Z$, solvability in the finite-energy
space requires the condition 
$$\int_\Omega ZE_{1,\lambda}=0.$$ 
This is precisely the orthogonality condition that determines $\lambda$.

For ease of notation, we define the positive profile integrals
\[
 I=\int_a^{a+1}\rho v(\rho)^2\dd\rho,\qquad
 J=\int_a^{a+1}\rho q(\phi(\rho))\dd\rho.
\]
Then, we prove the following:
\begin{lemma}\label{s:lem:Melnikov}
For every $\lambda>0$,
\begin{equation}\label{s:eq:Melnikov-exact}
 \int_\Omega Z E_{1,\lambda}\dd x\dd z
 =\pi I-4\pi\lambda J.
\end{equation}
The pairing has a unique positive zero, with derivative $-4\pi J\ne0$.
\end{lemma}
\begin{proof}
Since $x=\rho\cos\theta$ and $Z=v\cos\theta$, integration in $\theta$ gives
\[
 \int_\Omega Z E_{1,\lambda}
 =\pi\int_a^{a+1}\rho v^2\dd\rho
  +2\pi\int_a^{a+1}\rho^2v f(\phi)\dd\rho
  +2\pi\lambda\int_a^{a+1}\rho^2v q'(\phi)\dd\rho.
\]
Moreover, using $f(\phi(\rho))=v'(\rho)+v(\rho)/\rho$ and integrating by parts, we get that
\begin{align*}
 \int_a^{a+1}\rho^2v f(\phi)\dd\rho
 &=\left[\frac12\rho^2v^2\right]_a^{a+1}=0,\\
 \int_a^{a+1}\rho^2v q'(\phi)\dd\rho
 &=\left[\rho^2q(\phi)\right]_a^{a+1}-2J=-2J.
\end{align*}
The boundary terms vanish because $v$ is flat at the endpoints
and $q(0)=q(2)=0$. Substitution proves~\eqref{s:eq:Melnikov-exact}.
\end{proof}

The compatibility condition therefore selects
\[
 \lambda_*:=\frac I{4J}>0.
\]
We henceforth restrict $\lambda$ to
$\Lambda:=[\lambda_*/2,3\lambda_*/2]$. We will also use the radial quantity
\begin{equation}\label{s:eq:balance-defect}
 \mathfrak d_\lambda(\rho)=v(\rho)^2-4\lambda q(\phi(\rho)).
\end{equation}
Equivalently, the selected parameter satisfies the weighted balance
\begin{equation}\label{s:eq:integral-balance}
 \int_a^{a+1}\rho\,\mathfrak d_{\lambda_*}(\rho)\dd\rho=0.
\end{equation}
This integral balance will reappear in the proof of nonlocalizability.

\subsection{The compact equation and its first-order limit}

We seek $\psi=U+\varepsilon w$ with $w$ bounded and even in $z$.
The nonlinear inverse from Section~\ref{common:sec-compact} turns
this into a compact equation whose limit is the linear correction
equation. We work on the fixed real space
\[
\mathcal{X}:=\{w\in C_0(\overline D):w(x,-z)=w(x,z)\},\qquad
 \|w\|_{\mathcal{X}}=\|w\|_{L^\infty(D)}.
\]
Note that only the correction has zero boundary values, and so that $\psi$ has the boundary
values of $U$. All estimates below are uniform for $\lambda\in\Lambda$
and $0\le\varepsilon\le\varepsilon_0$.

By~\eqref{s:eq:endpoint-growth}, it follows that $\partial_tN_{\epsilon,\lambda} \ge1$ near $0$ and
$2$, uniformly in the parameters. Then, using
\eqref{common:eq-monotone-splitting} on $[0,2]$, with a fixed cutoff
$\xi$ equal to one on $[0,t_*]\cup[2-t_*,2]$ and supported in
$[0,2t_*)\cup(2-2t_*,2]$, we obtain $k_{\epsilon,\lambda}$ and $g_{\epsilon,\lambda}$. We also extend $k_{\epsilon,\lambda}$ constantly outside $[0,2]$.
Since we know that $N_{\epsilon,\lambda}(x,0)=N_{\epsilon,\lambda}(x,2)=0$, we get that
\begin{equation}\label{s:eq:abs-k-plateaus}
 k_{\varepsilon,\lambda}(x,t)=0,\quad t\le t_*,\qquad
 k_{\varepsilon,\lambda}(x,t)=g_{\varepsilon,\lambda}(x,2),
 \quad t\ge2-t_*.
\end{equation}
Next, we extend $g_{\epsilon,\lambda}$ by $g_{\epsilon,\lambda}(x,t)=t$ for $t<0$ and
$g_{\epsilon,\lambda}(x,t)=g_{\epsilon,\lambda}(x,2)+t-2$ for $t>2$. It is continuous, strongly monotone
and has at most linear growth. The smoothness and monotonicity of
$k_{\epsilon,\lambda}$, including smooth parameter dependence and uniform derivative
bounds, follow directly from the common construction.
Note that, at $\varepsilon=0$, $N_{0,\lambda}=f$ is independent of $x,\lambda$. We then
write the corresponding functions as $g_0=g_{0,\lambda}$ and $k_0 = k_{0,\lambda}$.

Now, we set $\mu_\varepsilon:=(1+\varepsilon x)^{-1}$, so that
$L_\varepsilon=\mu_\varepsilon^{-1}\operatorname{div}(\mu_\varepsilon\nabla)$.
For $F\in L^\infty(D)$, define $\mathcal T_{\varepsilon,\lambda}F=v$ by
\begin{equation*}
 -L_\varepsilon v+g_{\varepsilon,\lambda}(x,v)=F,
 \qquad v-U\in H_0^1(D).
\end{equation*}
The weights and their reciprocals have uniform smooth bounds, so
Proposition~\ref{common:prop-monotone-inverse} supplies these inverses,
their $W^{2,p}$ bounds, parameter continuity and collective compactness.
In particular,
\begin{equation}\label{s:eq:abs-inverse-Lipschitz}
 \|\mathcal T_{\varepsilon,\lambda}F-
        \mathcal T_{\varepsilon,\lambda}G\|_{L^\infty(D)}
 \le\|F-G\|_{L^\infty(D)}.
\end{equation}
Since
$0\le k_{\epsilon,\lambda}(x,t)\le g_{\epsilon,\lambda}(x,2)$, comparison with zero and two gives
\begin{equation}\label{s:eq:abs-range-preservation}
 0\le\mathcal T_{\varepsilon,\lambda}
       [k_{\varepsilon,\lambda}(x,v(x,z))]\le2,
\end{equation}
for every bounded measurable $v$. Thus, the continuations outside $[0,2]$
serve only to define the map. Its outputs stay in the physical range.

For the linear inverse, we take
\[
 V=g_0'(U)\ge1,\qquad \beta=k_0'(U),\qquad V-\beta=f'(U).
\]
The smooth function $\beta$ is supported in a compact subannulus of
$\Omega$, while $V\approx1+\log^2(1/d)/d^2$ near
$\partial\Omega$, with $d=\dist(\cdot,\partial\Omega)$.
Thus, Proposition~\ref{common:prop-singular-resolvent} gives
$\mathcal R=(-\Delta+V)^{-1}$ on the energy space
$\mathcal E=\mathcal E(\Omega,V)$, which agrees with
\eqref{s:eq:spec-planar-domain}. We extend its outputs by zero to $D$.
Both inverses preserve evenness in $z$ by uniqueness.

Next, we set
\begin{equation*}
 \mathcal K_\varepsilon(w,\lambda)
 =\frac{\mathcal T_{\varepsilon,\lambda}
           [k_{\varepsilon,\lambda}(x,U+\varepsilon w)]-U}{\varepsilon},
 \qquad w\in \mathcal{X}.
\end{equation*}
The boundary conditions and reflection symmetry ensure that this map
takes values in $\mathcal{X}$. However, identifying the limit as $\epsilon \downarrow 0$ is not so obvious. The candidate first-order limit is
\[
\mathcal K_0(w,\lambda) =\mathcal R[\beta w-E_{1,\lambda}].
\]
Formally, this corresponds to expanding the nonlinear fixed-point map at \(\varepsilon=0\) and \(\psi=U\). The following proposition justifies this expansion uniformly for bounded \(w\) and \(\lambda\in\Lambda\).

\begin{proposition}\label{s:thm:abs-scaled-limit}
For every $M_0<\infty$,
\begin{equation}\label{s:eq:abs-uniform-scaled-limit}
 \sup_{\substack{\norm{w}_{\mathcal{X}}\le M_0\\\lambda\in\Lambda}}
 \norm{\mathcal K_\varepsilon(w,\lambda)
           -\mathcal K_0(w,\lambda)}_{\mathcal{X}}\longrightarrow0,
               \quad \varepsilon\downarrow0.
\end{equation}
The family $\mathcal K_\varepsilon$ is continuous and collectively compact on
bounded subsets of $\mathcal{X}\times\Lambda$, for
$0\le\varepsilon\le\varepsilon_0$.
\end{proposition}

\begin{proof}
Fix $M_0$ and write
\[
 v= \mathcal T_{\varepsilon,\lambda}
       [k_{\varepsilon,\lambda}(x,U+\varepsilon w)],
 \qquad \zeta=\frac{v-U}{\varepsilon}.
\]
By~\eqref{s:eq:abs-range-preservation}, we know $0\le v\le2$. Moreover, observe that
\[
 U=\mathcal T_{\varepsilon,\lambda}
       [k_{\varepsilon,\lambda}(x,U)+E_{\varepsilon,\lambda}].
\]
The Lipschitz estimate~\eqref{s:eq:abs-inverse-Lipschitz} and
$\norm{E_{\varepsilon,\lambda}}_{L^{\infty}(D)}\le C\varepsilon$ therefore give
$\norm{\zeta}_{L^{\infty}(D)}\le C(M_0+1)=:C_0$.
Subtracting the two equations yields
\begin{equation*}
 -L_\varepsilon\zeta+c_{\varepsilon,\lambda}\zeta
       =\eta_{\varepsilon,\lambda,w},\qquad
 \zeta|_{\partial D}=0,
\end{equation*}
where $c_{\varepsilon,\lambda}\ge1$ is the secant slope of
$g_{\varepsilon,\lambda}(x,\cdot)$ between $v$ and $U$, and
\begin{equation}\label{s:eq:abs-eta-definition}
 \eta_{\varepsilon,\lambda,w}
 =\frac{k_{\varepsilon,\lambda}(x,U+\varepsilon w)
                    -k_{\varepsilon,\lambda}(x,U)}{\varepsilon}
       -\frac{E_{\varepsilon,\lambda}}{\varepsilon}.
\end{equation}
At coincident values in $(0,2)$ we use the derivative, and at coincident
endpoint values we set $c_{\varepsilon,\lambda}=1$. This does not
affect the product $c_{\varepsilon,\lambda}\zeta$, which is a bounded
function for each fixed $\varepsilon$. Taylor's formula and
\eqref{s:eq:reference-defect} then imply that
\begin{equation}\label{s:eq:abs-eta-expansion}
 \norm{\eta_{\varepsilon,\lambda,w}-\beta w+E_{1,\lambda}}_{L^{\infty}(D)}
 \le C_{M_0}\varepsilon.
\end{equation}

We now apply Proposition~\ref{common:prop-linearization} on the enlarged domain $D$, taking the active annulus to be $\Omega$, and the limiting potential to be $V = g_0'(U)$. The limiting inverse acts on $\Omega$ and its outputs are extended by zero to $D$.  On compact subsets of $\Omega$, the values of $U$
stay away from $0$ and $2$, so
$c_{\varepsilon,\lambda}\to V=g_0'(U)$ locally uniformly.
To check the collar assumption, fix $h>0$ so small that both two-sided
$2h$-collars lie in $D$, and take $\varepsilon C_0\le H(2h)$. In each collar,
the endpoint distances of both $U$ and $v$ lie in
$[0,2H(2h)]$. By~\eqref{common:eq-inverse-scale} and
the endpoint bound~\eqref{s:eq:endpoint-growth}, every nontrivial secant
satisfies
\[
 c_{\varepsilon,\lambda}\ge
 c\,\frac{\log^2(1/h)}{h^2},
\]
for some uniform constant $c > 0$.
Here,  $\partial_tg_{\epsilon,\lambda}=\partial_tN_{\epsilon,\lambda}$ near the endpoints, the factor
$A_\varepsilon$ is uniformly positive, and the term involving $q''$
is bounded. The same bound holds for a secant ending at $0$ or $2$
by integration of the locally integrable derivative.
On the plateaus $D\setminus\Omega$, both
$E_{\varepsilon,\lambda}$ and the numerator in
\eqref{s:eq:abs-eta-definition} vanish, the latter by
\eqref{s:eq:abs-k-plateaus}. Thus, the exterior source is exactly zero.

All hypotheses of Proposition~\ref{common:prop-linearization} now
hold along any sequence $\varepsilon_j\downarrow0$ with bounded
$w_j$ and $\lambda_j\in\Lambda$. Hence, it follows that
\[
 \norm{\mathcal K_{\varepsilon_j}(w_j,\lambda_j)
       -\mathcal R\eta_{\varepsilon_j,\lambda_j,w_j}}_{\mathcal{X}}\to0, \quad \textup{as } j \to \infty.
\]
Combining this with~\eqref{s:eq:abs-eta-expansion} and boundedness of
$\mathcal R$ proves~\eqref{s:eq:abs-uniform-scaled-limit}, by contradiction
if the convergence were not uniform on the indicated bounded set.
Compactness of $\mathcal R$ also gives collective compactness as
$\varepsilon\downarrow0$. For positive $\varepsilon$, continuity and
compactness follow from Proposition~\ref{common:prop-monotone-inverse}
and the smooth composition with $k_{\varepsilon,\lambda}$.
The uniform limit then gives continuity at $\varepsilon=0$.
\end{proof}

\subsection{Phase normalization and existence}

We fix the remaining translation by requiring orthogonality to $Z$.
Together with the free swirl parameter, this makes the limiting
equation invertible. More precisely, we set
\[
 \mathcal B w=\mathcal R(\beta w),\qquad
 \mathcal M=\Id-\mathcal B,\qquad
 P(w)=\int_\Omega Zw\dd x\dd z,
\]
and seek zeros of
\[
\mathcal{F}_\epsilon: \mathcal{X} \times \Lambda \to \mathcal{X} \times \RR, \qquad \mathcal F_\varepsilon(w,\lambda)
 =\bigl(w-\mathcal K_\varepsilon(w,\lambda),P(w)\bigr).
\]
Before going any further, let us point out that this map can be written as identity minus the compact map
$(w,\lambda)\mapsto(\mathcal K_\varepsilon(w,\lambda),\lambda-P(w))$. Then, we prove the following:

\begin{proposition}
\label{s:prop:degree-existence}
The affine map $\mathcal F_0$ has a unique zero
$(w_*,\lambda_*)\in \mathcal{X}\times\Lambda$.
For every sufficiently small $\varepsilon>0$, there are
$(w_\varepsilon,\lambda_\varepsilon)\in \mathcal{X}\times\operatorname{int}\Lambda$ such that
\begin{equation}\label{s:eq:degree-solutions}
 w_\varepsilon=\mathcal K_\varepsilon(w_\varepsilon,\lambda_\varepsilon),
 \qquad P(w_\varepsilon)=0,
\end{equation}
and they can be selected to satisfy
\begin{equation}\label{s:eq:w-lambda-limit}
 w_\varepsilon\to w_*\quad\text{in }\mathcal{X},\qquad
 \lambda_\varepsilon\to\lambda_*.
\end{equation}
Then, $\psi_\varepsilon=U+\varepsilon w_\varepsilon$ solves
\eqref{s:eq:GS-centered}, has the boundary values of $U$, and satisfies
$0\le\psi_\varepsilon\le2$ and
$\norm{\psi_\varepsilon-U}_{C(\overline D)}\le C\varepsilon$.
\end{proposition}

\begin{proof}
By Proposition~\ref{s:prop:spec-planar-kernel} and uniqueness of the
coercive inverse,
\[
 \ker\mathcal M=\operatorname{span}\{Z\},\qquad
 \mathcal R(\beta Z)=Z.
\]
Indeed, a fixed point of $\mathcal B$ belongs to $\mathcal E$, vanishes
outside $\Omega$, and solves $\mathcal Lw=0$. Now,  we define
$\Theta(y)=\int_\Omega\beta Zy\dd x\dd z$, and observe that self-adjointness of
$\mathcal R$ on $L^2(\Omega)$ gives
\[
 \Theta(\mathcal Mw)=0,\qquad
 \Theta(\mathcal R E)=\int_\Omega ZE\dd x\dd z.
\]
Then, we write $E_{1,\lambda}=E_{1,0}+\lambda E^{(q)}$, where
$E^{(q)}=2xq'(U)$, and stress that the map $\mathcal F_0$ is affine and
$$
\mathcal{F}_0(w,\lambda) =  \mathscr A(w,\lambda) + (\mathcal{R} E_{1,0},0),
$$
with
\[
 \mathscr A(w,\nu)=
 \bigl(\mathcal Mw+\nu\mathcal R E^{(q)},P(w)\bigr).
\]
Note that $\mathscr{A}$ is identity minus a compact operator on $\mathcal{X}\times\RR$, so its
Fredholm index is zero. If $\mathscr A(w,\nu)=0$, applying $\Theta$
and using~\eqref{s:eq:Melnikov-exact} gives $-4\pi J\nu=0$.
Thus $\nu=0$, $w$ is a multiple of $Z$, and $P(w)=0$ forces $w=0$.
Hence, $\mathscr A$ is bijective, and the bounded inverse theorem shows that it is boundedly invertible. The affine map
$\mathcal F_0$ has therefore a unique zero, whose compatibility condition gives
$\lambda=\lambda_*\in\operatorname{int}\Lambda$.

Next, choose a ball $\mathbb B$ of radius $r_0$ about $(w_*,\lambda_*)$ with
closure in $\mathcal{X}\times\operatorname{int}\Lambda$. On its boundary $\partial \mathbb{B}$, i.e. for $(w,\lambda)$ such that
$$
\|(w-w_*,\lambda-\lambda_*)\|_{\mathcal{X} \times \RR} = r_0,
$$
it follows that
\[
 \norm{\mathcal F_0(w,\lambda)}_{\mathcal{X} \times \RR}
 =\norm{\mathscr A(w-w_*,\lambda-\lambda_*)}_{\mathcal{X} \times \RR}
 \ge \frac{r_0}{\norm{\mathscr A^{-1}}_{\mathcal{L}(\mathcal{X}\times \RR)}} > 0.
\]
The uniform convergence in~\eqref{s:eq:abs-uniform-scaled-limit} then makes the straight homotopy from $\mathcal F_0$ to $\mathcal F_\varepsilon$, i.e.
$$
H_t(w,\lambda) = (1-t) \mathcal{F}_0(w,\lambda) + t \mathcal{F}_\epsilon(w,\lambda), \qquad 0 \leq t \leq 1,
$$
nonvanishing on $\partial\mathbb B$ for small enough $\varepsilon >0$. Its Leray--Schauder degree
is therefore the nonzero degree of the invertible affine map, so each
$\mathcal F_\varepsilon$ has a zero in $\mathbb B$. We choose one zero $(w_\epsilon, \lambda_\epsilon) \in \mathbb{B}$ of $\mathcal{F}_\epsilon$ for each $0 < \varepsilon \ll 1$. At these zeros, it follows that
\[
 \norm{\mathscr A(w_\varepsilon-w_*,\lambda_\varepsilon-\lambda_*)}_{\mathcal{X} \times \RR}
 =\norm{\mathcal F_0(w_\varepsilon,\lambda_\varepsilon)
       -\mathcal F_\varepsilon(w_\varepsilon,\lambda_\varepsilon)}_{\mathcal{X} \times \RR} \to0,
\]
which proves convergence of the selected family. 

Finally, observe that
\eqref{s:eq:degree-solutions} is precisely
$\psi_\varepsilon=\mathcal T_{\varepsilon,\lambda_\varepsilon}
[k_{\varepsilon,\lambda_\varepsilon}(x,\psi_\varepsilon)]$.
The range estimate~\eqref{s:eq:abs-range-preservation} gives
$0\le\psi_\varepsilon\le2$, where $g_{\epsilon,\lambda}-k_{\epsilon,\lambda}=N_{\epsilon,\lambda}$, and thus the original Grad--Shafranov equation and the remaining assertions.
\end{proof}

\subsection{Smooth localization and Euler reconstruction}\label{s:sec:localization}

We turn the solutions of Proposition~\ref{s:prop:degree-existence}
into global flows by creating constant collars at the artificial boundary
and applying the a priori regularity theorem at every contact with a
plateau. Write $d_-=a-\delta$ and $d_+=a+1+\delta$, so that
$D=\{d_-<\rho<d_+\}$. The inverse equation gives uniform
$W^{2,p}(D)$ bounds for every finite $p$, since its sources and solutions
are uniformly bounded. In particular, the solutions have uniform
$C^{1,\alpha}(\overline{D})$ bounds for every $\alpha\in(0,1)$.
Recall that the chosen Bernoulli derivative and normalization give
\[
 B_{\varepsilon,\lambda}(t)
 =\varepsilon^2\int_2^t\bigl(f(s)+\lambda q'(s)\bigr)\dd s,
\]
and note that this function belongs to $C^1([0,2])$, with derivative zero at both endpoints, and extends constantly outside that interval as a $C^1$ function.

\begin{proposition}\label{s:prop:re-euler}
For all sufficiently small $\varepsilon>0$, let
$\psi_\varepsilon$ be the solutions of Proposition~\ref{s:prop:degree-existence}. Then, there exists $\tau > 0$, independent of $\epsilon$, such that
\[
 \psi_\varepsilon=0,\quad d_-<\rho<d_-+\tau,\qquad
 \psi_\varepsilon=2,\quad d_+-\tau<\rho<d_+.
\]
Their extensions by $0$ to $\{\rho\le d_-\}$ and by $2$ to
$\{\rho\ge d_+\}$ are smooth, with all positive-order derivatives
vanishing on both plateau sets, and satisfy:
\begin{enumerate}
\item $ \psi_\varepsilon-2\in C_c^\infty(\RR^2)$ with supports contained in one fixed compact set. 
\item For every integer $q \geq 0$, $ \norm{\psi_\varepsilon-U}_{C^q(\RR^2)}\longrightarrow0$ as $\varepsilon\downarrow 0$.
\end{enumerate}
Furthermore, set $R=\varepsilon^{-1}$ and
$\Psi_\varepsilon(r,z)=\psi_\varepsilon(r-R,z)$. Then, the fields $u_\varepsilon,p_\varepsilon$ defined by
\eqref{s:eq:re-velocity}, with
$F=F_{\lambda_\varepsilon}$ and
$B=B_{\varepsilon,\lambda_\varepsilon}$, and extended by zero on the axis,
are smooth and compactly supported solutions to~\eqref{eq:euler}.
The velocities are axisymmetric with nonzero swirl, and their supports obey
\eqref{eq:spatial-support}.
\end{proposition}

\begin{proof}
\step{Exact collars.}
Since $q'(H(s))=-q'(2-H(s)) =O(H(s))$ for $0 < s \ll 1$, while both
$H''(s)/H'(s) \to \infty$ and $H''(s)/H(s) \to \infty$ as $s \downarrow 0$, the endpoint formulas give
\begin{align*}
 N_{\varepsilon,\lambda}(x,H(s))\ge c_0H''(s), \qquad -N_{\varepsilon,\lambda}(x,2-H(s))\ge c_0H''(s), \qquad 0 < s \ll 1,
\end{align*}
uniformly in $0 < \epsilon \ll 1$ and $\lambda \in \Lambda$. Moreover, both $N_{\epsilon,\lambda}(x,t)$ and $-N_{\epsilon,\lambda}(x,2-t)$ vanish at $t = 0$ and, for every $(x,z) \in \overline{D}$, have $t$-derivative at least one on a fixed interval $(0,t_0)$. Hence, \eqref{common:eq-exact-collar-absorption} holds uniformly on $\overline{D}$ for the functions $N_{\varepsilon,\lambda_\varepsilon}(x,t)$ and $-N_{\varepsilon,\lambda_\varepsilon}(x,2-t)$.

Next, we choose collars $C_{\pm}$ of $\partial D$ disjoint from $\overline\Omega$, set
$v_\varepsilon^{-}=\psi_\varepsilon$ in $C_{-}$ and
$v_\varepsilon^{+}=2-\psi_\varepsilon$ in $C_{+}$, and observe that they are nonnegative, have zero boundary trace, and satisfy $\norm{v_\varepsilon^{\pm}}_{L^{\infty}(C_{\pm})} \leq \norm{\psi_\epsilon-U}_{C(\overline{D})} \le C\varepsilon$.

Since the drift of $L_\varepsilon$ is uniformly bounded, and the curvatures of the corresponding boundary circles are uniformly bounded, Lemma~\ref{common:lem-exact-collar} implies that $v_{\epsilon}^{\pm} = 0$ in $C_{\pm}' \subset C_{\pm}$.

\step{Smooth extension and convergence.} We extend \(\psi_\varepsilon\) by \(0\) to \(\{\rho\le d_-\}\) and by \(2\) to \(\{\rho\ge d_+\}\), keeping the same notation. Since \(\psi_\varepsilon\) is constant in $C_{\pm}'$, and \(N_{\varepsilon,\lambda_\varepsilon}(x,0)=N_{\varepsilon,\lambda_\varepsilon}(x,2)=0\), the extension solves the same equation on a fixed bounded neighborhood \(D'\) of \(\overline D\).
The estimates~\eqref{s:eq:nearlinear-f} apply to \(f(t)\) and \(-f(2-t)\), and the polynomial terms \(q'(t)\) and \(-q'(2-t)\) satisfy the same estimates, since they vanish at \(t=0\). Moreover, the coefficients of the equation have uniform bounds in every \(C^m(\overline{D'})\) norm.

Applying Theorem \ref{common:lem-flat-regularity}  to \(\psi_\varepsilon\) near the plateau value \(0\), and to \(2-\psi_\varepsilon\) near the plateau value \(2\), gives smoothness and flatness at both plateau sets. The uniform \(C^{1,\alpha}(\overline{D'})\) bounds, which immediately follow from the uniform $C^{1,\alpha}(\overline{D})$ bounds, allow these neighborhoods to be chosen independently of \(\varepsilon\). Away from the plateau values, ordinary elliptic regularity applies. Thus, for every integer \(q\ge0\),
\[
\|\psi_\varepsilon\|_{C^q(\overline{D'})}\le C_q,
\]
and all positive-order derivatives vanish wherever \(\psi_\varepsilon=0\) or \(\psi_\varepsilon=2\).
The extensions consequently belong to \(C^\infty(\mathbb R^2)\) and
$
\operatorname{supp}(\psi_\varepsilon-2)
\subset\{\rho\le d_+\}.
$

Finally, by Proposition \ref{s:prop:degree-existence}, \(\psi_\varepsilon-U\to0\) uniformly on \(\mathbb R^2\) as $\epsilon \downarrow 0$. Together with the uniform \(C^{q+1}(\overline{D'})\) bounds and compactness, this gives
\[
\|\psi_\varepsilon-U\|_{C^q(\mathbb R^2)}
\longrightarrow0
\quad\text{as }\varepsilon\downarrow0, \quad \textup{for every integer } q \geq 0.
\]

\step{The Euler fields.}
Since $F_{\lambda_\epsilon}(0)=F_{\lambda_\epsilon}(2)=0$, the smooth extension and constant collars give a smooth velocity supported
strictly inside the toroidal annulus in~\eqref{eq:spatial-support}.
Moreover, the region where $\Psi_\varepsilon$ differs from $2$ lies in
$r>R-d_+>0$. Hence, the velocity vanishes near the axis and its
cylindrical formula extends smoothly there. On the other hand, the uniform convergence to $U$ as $\epsilon \downarrow 0$ gives points with
$0<\Psi_\varepsilon<2$.  The swirl is therefore nonzero at those points because
$\lambda_\varepsilon>0$.

By the choice of the scalar data, equation~\eqref{s:eq:GS-centered}
is exactly~\eqref{s:eq:GS-physical}. The reconstruction in
Subsection~\ref{s:sec:scalar-reduction} therefore gives \eqref{eq:euler} with
the pressure given in \eqref{s:eq:re-velocity}.  

Finally, observe that the pressure is initially $C^1(\RR^3)$ because $B_{\epsilon,\lambda}$ is $C^1$. Using \eqref{eq:euler}, we see that $\nabla p_\epsilon=-(u_\epsilon\cdot\nabla)u_\epsilon\in C^\infty(\RR^3; \RR^3)$, so $p_\epsilon$ is actually smooth. Also, outside the bounded solid torus \(\{(r-R)^2+z^2\le d_+^2\}\), we have \(\Psi_\varepsilon=2\) and \(u_\varepsilon=0\). The normalization \(B_{\varepsilon,\lambda_\varepsilon}(2)=0\) therefore gives \(p_\varepsilon=0\) there, proving compact support.
\end{proof}

\subsection{The obstruction to localizability}

We compare the speed at two points on a regular stream-function
level. Since the Bernoulli function is constant there, a difference in speed forces a difference in pressure. The correction $w_*$ from
Proposition~\ref{s:prop:degree-existence} supplies the first-order
variation needed for this comparison.

\begin{proposition}\label{s:prop:nl-euler-field}
For every sufficiently small $\varepsilon>0$,
$u_\varepsilon\cdot\nabla p_\varepsilon\not\equiv0$.
\end{proposition}

\begin{proof}
\step{The first-order correction.}
On compact subsets of $\Omega$, the nonlinearity is smooth. Hence, combining \eqref{s:eq:GS-centered} and  \eqref{s:eq:reference-defect},
and applying interior elliptic estimates upgrades
\eqref{s:eq:w-lambda-limit} to $w_\varepsilon\to w_*$ in every
$C^q_{\loc}(\Omega)$ norm.
Its limiting equation is
\begin{equation}\label{s:eq:nl-limit-linear-equation}
 (\Delta-f'(U))w_*=U_x+2x\bigl(f(U)+\lambda_*q'(U)\bigr).
\end{equation}
Since \(\mathcal L\) has radial coefficients and the forcing contains only the first cosine mode, uniqueness of the solution satisfying \(\int_\Omega Zw_*=0\), established in Proposition \ref{s:prop:degree-existence}, implies that \(w_*=b(\rho)\cos\theta\) with
\begin{equation}
 b''+\frac{b'}\rho-\frac b{\rho^2}-f'(\phi)b =v+2\rho\bigl(f(\phi)+\lambda_*q'(\phi)\bigr).
 \label{s:eq:nl-b-ode}
 \end{equation}
Let us also recall here that, for $v = \phi'$,
 \begin{equation}
 v''+\frac{v'}\rho-\frac v{\rho^2}-f'(\phi)v=0.
 \label{s:eq:nl-v-ode}
\end{equation}
Having these equations at hand, we introduce $A = vb' - v'b - \rho (v^2 + 2\lambda_* q(\phi))$ and point out that, subtracting $b$ times~\eqref{s:eq:nl-v-ode} from $v$ times
\eqref{s:eq:nl-b-ode}, and using $f(\phi)=v'+v/\rho$, yields
\begin{equation}\label{s:eq:speed-balance}
 A'+\frac A\rho=v^2-4\lambda_*q(\phi)=\mathfrak d_{\lambda_*}.
\end{equation}
The compatibility condition used in Proposition \ref{s:prop:degree-existence} gives the integral balance
\[
\int_a^{a+1}\rho\bigl(v^2-4\lambda_*q(\phi)\bigr)\dd\rho=0.
\]

If \(A\) vanished identically on \((a,a+1)\), equation \eqref{s:eq:speed-balance} would instead imply
\[
v^2=4\lambda_*q(\phi)
\qquad\text{throughout }(a,a+1).
\]
Since \(v=\phi'>0\), \(0<\phi<2\), and \(q(\phi)=\phi^2(2-\phi)^2\), this gives
\[
\phi'=2\sqrt{\lambda_*}\,\phi(2-\phi).
\]
The equation extends continuously to \(\rho=a\), where \(\phi(a)=0\). Uniqueness for this scalar initial-value problem forces \(\phi\equiv0\), contradicting its strict increase on \((a,a+1)\). Hence, there exists \(\rho_0\in(a,a+1)\) such that \(A(\rho_0)\ne0\).

\step{Speed along a stream-function level.}
We fix $\rho_0 \in (a,a+1)$ such that $A(\rho_0) \neq 0$ and define the scaled squared speed in centered coordinates
\[
 S_\varepsilon=
 \frac{|\nabla\psi_\varepsilon|^2+
                2\lambda_\varepsilon q(\psi_\varepsilon)}
      {(1+\varepsilon x)^2},
 \qquad |u_\varepsilon|^2=R^{-2}S_\varepsilon,
 \quad R=\varepsilon^{-1},
\]
Since $v(\rho_0)>0$, the level curve
$\psi_\varepsilon=\phi(\rho_0)$ has a regular closed component
near the circle of radius $\rho_0$, parameterized in centered polar
coordinates by
$(\rho_\varepsilon(\theta),\theta)$, with
\[
 \rho_\varepsilon(\theta)
 =\rho_0-\varepsilon\frac{b(\rho_0)}{v(\rho_0)}\cos\theta
       +o(\varepsilon).
\]
Here and in the next expansion the remainders are uniform in $\theta$.
Along this level curve, the swirl term in the numerator of \(S_\varepsilon\) is exactly \(2\lambda_\varepsilon q(\phi(\rho_0))\). Thus, expanding as well the gradient and the denominator gives
\begin{equation}\label{s:eq:speed-on-level}
 S_\varepsilon(\rho_\varepsilon(\theta),\theta)
 =G_\varepsilon(\rho_0)+2\varepsilon A(\rho_0)\cos\theta
       +o(\varepsilon),
\end{equation}
where
\[
 G_\varepsilon=v^2+2\lambda_\varepsilon q(\phi),\qquad
 A=vb'-v'b-\rho\bigl(v^2+2\lambda_*q(\phi)\bigr).
\]
The term $vb'$ comes from the gradient correction, $-v'b$ from
the displacement of the level curve, and the last term from the denominator. Let us explicitly point out that keeping
$\lambda_\varepsilon$ in $G_\varepsilon$ requires only its convergence,
not any estimate for $(\lambda_\varepsilon-\lambda_*)/\varepsilon$.

By~\eqref{s:eq:speed-on-level}, the values of $S_\varepsilon$ at $\theta=0$ and
$\theta=\pi$ on the same level curve $\psi_\varepsilon=\phi(\rho_0)$ differ
by $4\varepsilon A(\rho_0)+o(\varepsilon)$, and hence are distinct
for every sufficiently small $\varepsilon>0$. Since the Bernoulli function
is constant on this level curve, the pressure is not. Moreover, observe that the meridional velocity
\[
u_{\varepsilon,\mathrm{mer}}
=\frac1r\left(-\partial_z\Psi_\varepsilon\,e_r
+\partial_r\Psi_\varepsilon\,e_z\right),
\]
is tangent to the level curve and does not vanish there. If \(u_{\varepsilon,\mathrm{mer}}\cdot\nabla p_\varepsilon\) vanished everywhere on this curve, the tangential derivative of \(p_\varepsilon\) would vanish everywhere, forcing the pressure to be constant along the curve, and thus contradicting the pressure difference established above. Finally, axisymmetry gives \(\partial_\varphi p_\varepsilon=0\), so the azimuthal component contributes nothing to \(u_\varepsilon\cdot\nabla p_\varepsilon\). Therefore, we conclude that
$
u_\varepsilon\cdot\nabla p_\varepsilon
=u_{\varepsilon,\mathrm{mer}}\cdot\nabla p_\varepsilon
\not\equiv0.
$
\end{proof}

\begin{proof}[Proof of Theorem~\ref{thm:spatial-precise}]
Fix $a$ as in Proposition~\ref{s:prop:spec-planar-kernel} and
$\delta\in(0,1/8)$. Propositions~\ref{s:prop:degree-existence}--\ref{s:prop:nl-euler-field}
give the required fields and all their support and convergence properties
for every sufficiently small $\varepsilon>0$. The identity
$\psi_\varepsilon-U=\varepsilon w_\varepsilon$, with $w_\varepsilon$
uniformly bounded, gives the $O(\varepsilon)$ estimate. Set
$R=\varepsilon^{-1}$ and choose $R_*>a+2$ beyond all smallness
thresholds. This proves the theorem for every $R>R_*$, and fixing one
such $R$ gives Theorem~\ref{thm:spatial}.
\end{proof}

\section{Non-locally-radial planar Euler flows}\label{p:sec-reference}
We construct the planar flows by adapting the annular construction of \cite{EFR} to the framework of Sections \ref{common:sec-prelim} and \ref{common:sec-compact}. Starting from a radial profile that rises from zero and then returns to zero, we prescribe separate nonlinearities on its increasing and decreasing branches. Varying the annulus radius changes the local degree and yields a sequence of nonradial solutions converging to a radial reference profile. The localization and regularity results then give smooth compactly supported flows, and the leading angular deformation rules out local radiality.
 
We use the convention
\begin{equation}\label{eq:planar-convention}
 \nabla^\perp\psi=(\partial_{x_2}\psi,-\partial_{x_1}\psi),
 \qquad u=\nabla^\perp\psi,\qquad \omega=\curl u=-\Delta\psi,
\end{equation}
and, for all $a > 3$, set
\[
 \Omega_a=\{x\in\RR^2:a-3<|x|<a+3\},\qquad r_-=a-3,\quad r_+=a+3.
\]

\begin{theorem}\label{thm:planar-precise}
There exist an integer $\ell\geq3$, a number $a_*>4$, a nonzero radial
function $\psi_*\in C_c^\infty(\RR^2)$ and nonnegative functions
$\psi_j\in C_c^\infty(\RR^2)$ with the following properties:
\begin{enumerate}
\item Each $\psi_j$ satisfies
\begin{equation*}
 \nabla^\perp\psi_j\cdot\nabla\Delta\psi_j=0
 \quad\textup{in }\RR^2.
\end{equation*}
Consequently, $u_j=\nabla^\perp\psi_j$ solves~\eqref{eq:euler} in dimension two with a
pressure $p_j\in C_c^\infty(\RR^2)$ normalized to vanish at infinity. \smallbreak
\item The functions $\psi_j$ are not locally radial. They are invariant
under reflection across the horizontal axis and have exactly $\ell$-fold
rotational symmetry about the origin. \smallbreak

\item The supports of all the $\psi_j$ lie in one fixed compact annulus,
the function $\psi_*$ is positive precisely on $\Omega_{a_*}$, and, for every integer $q \geq 0$, $\psi_j \to \psi_*$ in $C^q(\RR^2)$ as $j \to \infty$. 
\end{enumerate}
\end{theorem}

\begin{remark} $ $
\begin{itemize}
\item[(a)] The vorticities have zero moments against every harmonic
polynomial. This follows for any $\psi\in C_c^\infty(\RR^2)$ from
integration by parts:
\begin{equation*} 
 \int_{\RR^2}P\,(-\Delta\psi)\dd x
 =-\int_{\RR^2}\psi\,\Delta P\dd x=0,
 \qquad \Delta P=0.
\end{equation*}
\item[(b)]  All parameters below are chosen independently of those in Section \ref{s:sec:profile}.
\end{itemize}
\end{remark}

\subsection{The radial family and the stationary equation}
Each intermediate value of the profile occurs on two branches, so
we do not impose a single law for its Laplacian. Following~\cite{EFR},
we make the two laws equal to $-1$ above a common threshold and
interpolate only at radii where nearby solutions stay above that
threshold. This removes explicit radial dependence along their graphs.

To obtain a region where the two laws agree, we start with a profile
of constant Laplacian and modify it only near the endpoints to make
it flat. Let $\tau_a$ be the unique solution to
\begin{equation*}
 \tau_a''+\frac1r\tau_a'=-1,
 \qquad \tau_a(r_-)=\tau_a(r_+)=0,
\end{equation*}
which has an explicit formula:
\begin{equation}\label{p:eq:torsion-explicit}
 \tau_a(r)=C_a\log\frac r{r_-}-\frac{r^2-r_-^2}{4},
 \qquad
 C_a=\frac{r_+^2-r_-^2}{4\log(r_+/r_-)}.
\end{equation}
The formula gives $\tau_a''(r)=-C_a/r^2-1/2<0$, so its maximum is
unique and nondegenerate. Moreover, expanding in $a^{-1}$ gives, in the shifted coordinate $x=r-a$,
\begin{equation}\label{p:eq:torsion-limit}
 \tau_a(a+x)\longrightarrow\bar\tau(x):=\frac{9-x^2}{2}
 \quad\hbox{in }C^q([-3,3]),
\end{equation}
for each fixed $q$, with error $O_q(a^{-1})$.

Having $\tau_a$ at hand, we can now introduce our radial profile $U_a \in C_c^{\infty}(\RR^2)$. In the boundary collars we use the flat profile $H$
introduced in \eqref{common:eq-H}, and in the interior $\tau_a$. More precisely, let $\zeta$ be a smooth nonincreasing cutoff function equal to one for $s\le1/4$ and zero for $s\ge1/2$. In the inner collar $0<s=r-r_-\le1/2$, we set
\begin{equation*}
 U_a(r)=\zeta(s)H(s)+(1-\zeta(s))\tau_a(r),\qquad s=r-r_-.
\end{equation*}
In the outer collar $0 < s = r_+-r \le1/2$, we use the identical formula with $s=r_+-r$, and  in the rest of the annulus we set  $U_a=\tau_a$. Finally, we extend $U_a$ by zero outside $\Omega_a$. 

Two margins justify this modification. On the transition intervals
$1/4\le s\le1/2$, the torsion profile dominates $H$, so the cutoff
preserves strict monotonicity on each branch. Throughout the modified
collars, both profiles have values below two, leaving a gap below the
common threshold three. By~\eqref{p:eq:torsion-limit}, the strict
inequalities
\[
 \min_{1/4\le s\le1/2}\bar\tau(3-s)>H(1/2),\qquad
 \max_{0\le s\le1/2}\bar\tau(3-s)<2
\]
give both margins for large $a$.

Thus, for some $A_0>4$ and every $a\ge A_0$, the function
$U_a\in C_c^\infty(\RR^2)$ is positive precisely on $\Omega_a$ and
has a unique critical radius in $(r_-,r_+)$, at the torsion maximum, where
$U_a''=-1$. The family depends smoothly on $a$. By
\eqref{p:eq:torsion-limit}, its shifted profiles converge in every
$C^q([-3,3])$ norm to an even function $\bar U$, positive in $(-3,3)$,
with a unique nondegenerate maximum at zero and equal to $H(3-|x|)$
near the endpoints. Increasing $A_0$ if necessary, we also have
$U_a>3+\sigma$ on $|r-a|\le1$ for a fixed $\sigma>0$.

The radial Laplacian can differ at the two radii where $U_a$ takes
one value. For $0<t<3$, we invert the increasing and decreasing branches
separately, writing their inverses as $r_{\mathrm i,a}(t)$ and
$r_{\mathrm o,a}(t)$, and define the functions
\begin{equation}\label{p:eq:f-by-inversion}
 F_{\iota,a}(t)
 =U_a''(r_{\iota,a}(t))
  +\frac{U_a'(r_{\iota,a}(t))}{r_{\iota,a}(t)},
 \quad\iota\in\{\mathrm i,\mathrm o\}.
\end{equation}
Since the modification is confined to values below two, it follows that
$F_{\iota,a}=-1$ near $t=3$. We then extend it by $F_{\iota,a}(t)=-1$ for
$t\ge3$ and $F_{\iota,a}(0)=0$. The resulting
functions are smooth for $t>0$ and depend smoothly on $a$. Both laws
become constant before the profile reaches its maximum, so no inversion
at the critical radius is needed.

Now, let $\chi\in C^\infty(\RR)$ be nondecreasing, zero on $(-\infty,-1]$, and one on $[1,\infty)$. We define
\begin{equation}\label{p:eq:f-full}
 F_a(r,t)=(1-\chi(r-a))F_{\mathrm i,a}(t)
          +\chi(r-a)F_{\mathrm o,a}(t),\qquad t\ge0.
\end{equation}
Then, after zero extension, it follows that
\begin{equation}\label{p:eq:radial-pde}
 \Delta U_a=F_a(|x|,U_a) \quad \textup{in } \RR^2.
\end{equation}
Let us however point out that the formula is only needed on annuli bounded away from the origin.  Near the origin the radial profile is identically zero.
 
\begin{lemma}\label{p:lem:autonomy}
Fix a compact interval $J\subset[A_0,\infty)$ and an annulus $D$
containing all $\overline\Omega_a$, $a\in J$. There exists $\rho_0>0$ such
that, if $a\in J$, $w\ge0$ and
$\|w-U_a\|_{C(\overline D)}<\rho_0$, then
\begin{equation}\label{p:eq:autonomy}
 \partial_rF_a(|x|,w(x))=0,\qquad x\in D.
\end{equation}
\end{lemma}
\begin{proof}
For every fixed $t \geq 0$, it follows that
\[
 \partial_rF_a(r,t)=\chi'(r-a)
       [F_{\mathrm o,a}(t)-F_{\mathrm i,a}(t)].
\]
On the support of $\chi'$, we have that $U_a>3+\sigma$. Taking
$\rho_0<\sigma/2$ makes both laws equal to $-1$ at $t=w(x)$, and the result follows. 
\end{proof}

This is precisely the lemma allowing us to use non-autonomous nonlinearities. It essentially says that, if we obtain a nonnegative smooth solution to $\Delta w = F_a(|x|,w)$ sufficiently close to $U_a$, we can go back and obtain a solution to \eqref{eq:euler}. 

We now record the endpoint estimates needed for the common analytic
results. In the two collars,
\begin{equation}\label{p:eq:endpoint-laws}
 F_{\mathrm i,a}(H(s))=H''(s)+\frac{H'(s)}{r_-+s},\qquad
 F_{\mathrm o,a}(H(s))=H''(s)-\frac{H'(s)}{r_+-s}.
\end{equation}
Differentiating as in~\eqref{s:eq:fprime-inner}, with the opposite
curvature sign on the decreasing branch, gives
\begin{equation}\label{p:eq:endpoint-growth}
 F'_{\mathrm i,a}(H(s))\approx F'_{\mathrm o,a}(H(s))
 \approx\frac{\log^2(1/s)}{s^2},
\end{equation}
uniformly for large $a$. The value estimates~\eqref{common:eq-value-calculus}
also give, for all integers $b,j\ge0$ and every $\eta>0$,
\begin{equation}\label{p:eq:nearlinear-collar}
 |\partial_a^b\partial_t^jF_{\iota,a}(t)|
 \le C_{b,j,\eta}t^{1-j-\eta},\qquad 0<t<t_c,
 \quad\iota\in\{\mathrm i,\mathrm o\},
\end{equation}
uniformly on compact parameter intervals.
The growth bound~\eqref{p:eq:endpoint-growth} implies that both slopes
$F'_{\iota,a}(t)$ tend to $+\infty$ as $t\downarrow0$, uniformly for
large $a$.
Moreover, for  $J$ and $D$ as in Lemma~\ref{p:lem:autonomy}, the fixed spatial cutoff in~\eqref{p:eq:f-full} then yields, for every multi-index $\alpha$,
integer $j\ge0$ and $\eta>0$,
\begin{equation}\label{p:eq:F-nearlinear}
 |D_x^\alpha\partial_t^jF_a(|x|,t)|
 \le C_{\alpha,j,\eta}t^{1-j-\eta},\qquad
 0<t<t_c,\quad x\in\overline D,\ a\in J.
\end{equation}
The same bounds hold with any fixed number of parameter derivatives.

\subsection{Spectral separation}\label{p:sec-spectral}
We need an angular mode whose first eigenvalue changes sign while
all other modes in the chosen symmetry class stay away from zero.
The single zero of $U_a'$ supplies a negative direction below the
translation mode. The positive term $n^2/r^2$ makes the first
eigenvalue positive at a fixed radius for high angular frequency;
increasing the radius while keeping that frequency fixed eventually
makes it negative. We isolate this change in the even, rotationally
symmetric space, so that any kernel encountered is one-dimensional.

We set
\begin{equation*}
 P_a(r)=\partial_t F_a(r,U_a(r)),\qquad r_-<r<r_+.
\end{equation*}
The collar bound~\eqref{p:eq:endpoint-growth} supplies a positive
smooth radial weight $W_a$ in the open annulus
which equals $1+\log^2(1/d_a)/d_a^2$ for sufficiently small
$d_a=\dist(x,\partial\Omega_a)$ and is bounded above and below away from the
boundary. On a compact parameter interval, it then follows that
\begin{equation}\label{p:eq:P-weight}
 cW_a-C\le P_a\le CW_a+C.
\end{equation}
The angular-frequency-$n$ quadratic form is
\begin{equation*}
 \mathfrak q_{a,n}(q)=\int_{r_-}^{r_+}
 \left((q')^2+\left(P_a(r)+\frac{n^2}{r^2}\right)q^2\right)r\dd r,
 \qquad n\ge0,
\end{equation*}
on $\mathcal E_a^{\mathrm{rad}}=\mathcal E((r_-,r_+),W_a)$.
As in the three-dimensional construction, the weight bound and
Lemma~\ref{common:lem-density} give the finite-energy realizations
\[
 \mathcal L_{a,n}=-\partial_r^2-\frac1r\partial_r+P_a(r)+\frac{n^2}{r^2}
\]
in $L^2((r_-,r_+),r\dd r)$, with compact resolvent. We denote their
eigenvalues, in increasing order and with multiplicities, by $\mu_j(a,n)$, $j \in \NN$. Likewise, we denote by \(\mathcal L_a=-\Delta+\partial_tF_a(|x|,U_a)\) the operator associated with the full quadratic form. Note that its restriction to angular frequency \(n\) is \(\mathcal L_{a,n}\).

\begin{proposition}
\label{p:prop:spectral-input}
There exist an integer $\ell\ge3$ and parameters $a_1>a_0\ge A_0$
such that, for every $a \in J = [a_0,a_1]$, it follows that:
\begin{enumerate}[label=\textup{(\roman*)}]
\item $\mu_1(a,0) < \mu_2(a,0) < 0 < \mu_3(a,0)$. 
\item $\mu_2(a,\ell) > 0$ and $\mu_1(a,k\ell) > 0$ for every $k \geq 2$. 
\item $\mu_1(a_0,\ell) > 0$ and $\mu_1(a_1,\ell) < 0$. 
\end{enumerate}
If \(\mu_1(a,\ell)=0\), the kernel in the even, \(\ell\)-fold symmetric space is
\[
\ker\mathcal L_a
=\operatorname{span}\{q_a(r)\cos(\ell\theta)\},
\qquad q_a>0\text{ on }(a-3,a+3).
\]
Otherwise, the kernel is trivial. In particular, in this space, the endpoint operators $\mathcal{L}_{a_0}$ and $\mathcal{L}_{a_1}$ are invertible and have Morse indices \(2\) and \(3\), respectively.
\end{proposition}
\begin{proof}
\step{The translational mode and its negative direction.}
First of all, differentiating~\eqref{p:eq:radial-pde} with respect to $r$, and  using $\partial_r F_a(r,U_a)=0$, we get that
\begin{equation*}
 \mathcal L_{a,1}U_a'=0.
\end{equation*}
The function $U_a'$ is in $\mathcal E_a^{\mathrm{rad}}$ because it is infinitely flat at the endpoints. Moreover, it has exactly one zero in the interior, at the maximum $r_m$.

On each side of the maximum $r_m$, the positive function
$y=|U_a'|$ solves the homogeneous equation, so integration by parts gives
\begin{equation*}
 \mathfrak q_{a,1}(q)=\int_{r_-}^{r_+} ry^2\left|\left(\frac qy\right)'\right|^2\dd r\ge0
\end{equation*}
for compactly supported $q$ in that half-interval. This extends to
form-domain functions with $q(r_m)=0$ by density: use
Lemma~\ref{common:lem-density} at the singular endpoints and the
usual zero-trace approximation at the regular interior point $r_m$.
The form is nonnegative on this codimension-one subspace, so it has
at most one negative direction. On the other hand, it has at least one: otherwise $|U_a'|$, whose form value is zero,
would belong to the weak kernel, whereas its derivative has a nonzero
jump at $r_m$ because $U_a''(r_m)=-1$.

For two finite-energy solutions $y,z$ of the same eigenvalue
equation, the weighted Wronskian $\mathscr W(r)=r(y(r)z'(r)-y'(r)z(r))$ is constant.
If we set  $I_\delta=(r_-,r_-+\delta)$,  integrating the constant Wronskian
and using the zero traces, we get that
\[
 \delta|\mathscr W|
 =\left|\int_{I_\delta}r(yz'-y'z)\dd r\right|
 \le C\delta\|y'\|_{L^2(I_\delta)}\|z'\|_{L^2(I_\delta)}.
\]
Dividing by $\delta$ gives
\[
 |\mathscr W|\le C\|y'\|_{L^2(r_-,r_-+\delta)}
                    \|z'\|_{L^2(r_-,r_-+\delta)}\longrightarrow0, \quad \textup{as } \delta \downarrow 0.
\]
Thus $\mathscr W=0$, and ODE uniqueness makes the solutions
proportional. This proves simplicity of every eigenvalue.
Together with
compactness of the resolvent, this proves, for every $a\ge A_0$,
\[
 \mu_1(a,1)<0,\qquad \mu_2(a,1)=0,\qquad \mu_3(a,1)>0,
 \qquad \ker\mathcal L_{a,1}=\operatorname{span}\{U_a'\}.
\]

\step{The large-radius limit.}
Since $\mathcal L_{a,n}-\mathcal L_{a,1}=(n^2-1)/r^2$, min--max gives
\begin{equation*}
 \mu_2(a,n)\ge\frac{n^2-1}{(a+3)^2}>0,\qquad n\ge2,
\end{equation*}
and
\[
 \mu_2(a,0)\le-\frac1{(a+3)^2}<0.
\]

We next let $a$ tend to infinity. The limit will give a positive gap
above the translation mode, excluding further negative radial eigenvalues,
and a negative first eigenvalue that persists in each fixed angular
frequency as its angular energy $n^2/r^2$ tends to zero.

In $x=r-a\in(-3,3)$, let $\bar P$ be the limit of $P_a(a+x)$. Note that it equals $H'''/H'$ in the endpoint collars and that, on compact subintervals, it is the smooth limit obtained from~\eqref{p:eq:f-by-inversion}. Also, near the maximum both the finite-$a$ and limiting nonlinearities are identically $-1$.
The limit form is
\[
 \overline{\mathfrak q}(q)=\int_{-3}^3\bigl((q')^2+\bar Pq^2\bigr)\dd x,
\]
on the fixed space $ \mathcal E_\infty=H_0^1(-3,3)\cap L^2(W(x)\dd x),$ with
\[
W\approx1+\frac{\log^2(1/(3-|x|))}{(3-|x|)^2},
\]
near the endpoints, with a harmless positive extension to the interior. Also, note that the limiting derivative $\bar U'$ solves
$(-\partial_x^2+\bar P)\bar U'=0$. The preceding argument with  the radial
weight replaced by one then shows that
\begin{equation*}
 \bar\mu_1<0,\qquad\bar\mu_2=0,\qquad\bar\mu_3>0.
\end{equation*}

In the coordinate $x=r-a$, divide the form and its mass inner product
by $a$, and denote the resulting form on $\mathcal E_\infty$ by
$\mathfrak q^{\mathrm{shift}}_{a,n}$. Their density is
$w_a(x)=1+x/a\to1$ uniformly. The endpoint formulas give
\begin{equation*}
 |P_a(a+x)-\bar P(x)|
 \le\frac Ca\left(1+\frac{\log(1/d)}d\right),\qquad d=3-|x|,
\end{equation*}
and the interior difference is $O(a^{-1})$. Since
$\log(1/d)/d\le CW$ and $n^2/(a+x)^2\to0$ uniformly for fixed $n$, we then have that
\begin{equation*}
 |\mathfrak q^{\mathrm{shift}}_{a,n}(q)-\overline{\mathfrak q}(q)|
 \le e_a\int_{-3}^3((q')^2+Wq^2)\dd x,\qquad e_a\to0.
\end{equation*}
The mass inner products converge as well. Add a fixed multiple of the
$L^2$ norm so that the limiting form is coercive on $\mathcal E_\infty$;
the shifted forms are then uniformly equivalent to it for large $a$.
Min--max therefore gives, for each fixed $n\ge0$ and $j\ge1$,
\begin{equation}\label{p:eq:spectral-convergence}
 \mu_j(a,n)\longrightarrow\bar\mu_j,\qquad a\to\infty.
\end{equation}
Translating the endpoints to $(-3,3)$ also gives a common form domain
near any fixed parameter. The same estimate and smooth parameter
dependence prove that $\mu_j(a,n)$ is continuous in $a$ on finite
parameter intervals.
Since $\mu_3(a,0)\to\bar\mu_3>0$ and $\mu_2(a,0)<0$ by the
bound above, increasing $A_0$ makes the radial sector have exactly two
negative eigenvalues and no zero eigenvalue for all $a\ge A_0$.

\step{Choosing and isolating the angular frequency.}
Choose $a_0\ge A_0$. Increase an integer $\ell\ge3$ until
\[
 \mu_1(a_0,\ell)>0.
\]
Indeed
$\mu_1(a_0,\ell)\ge\mu_1(a_0,0)+\ell^2/(a_0+3)^2$.
With this $\ell$ fixed,~\eqref{p:eq:spectral-convergence} gives
$\mu_1(a,\ell)\to\bar\mu_1<0$.
Set
\[
 a_c=\inf\{a\ge a_0:\mu_1(a,\ell)<0\}.
\]
Then $a_c>a_0$, is finite, and $\mu_1(a_c,\ell)=0$.
For every $k\ge2$, min--max gives
\begin{equation*}
 \mu_1(a,k\ell)\ge\mu_1(a,\ell)
                 +\frac{(k^2-1)\ell^2}{(a+3)^2}.
\end{equation*}
For $a\in[a_0,a_c]$ all these higher frequencies are positive.
The definition of $a_c$ gives parameters with $\mu_1(a,\ell)<0$
arbitrarily close to $a_c$ from the right. Continuity keeps the
right-hand side for $k=2$ positive in a neighborhood of $a_c$.
Choose $a_1>a_c$ in that neighborhood with $\mu_1(a_1,\ell)<0$;
then every higher frequency stays positive throughout $[a_c,a_1]$. Set
\begin{equation*}
 J=[a_0,a_1].
\end{equation*}
These inequalities give (i), (ii) and (iii), and the claimed Morse indices. 

For the characterization of the kernel, suppose that \(\mu_1(a,\ell)=0\). Choose a normalized minimizer \(q_a\) of the first Rayleigh quotient. Replacing it by its absolute value preserves the quotient, so we may assume \(q_a\ge0\). If \(q_a\) vanished at an interior point, its derivative would also vanish there, and uniqueness for the eigenvalue ODE would force \(q_a\equiv0\). Thus \(q_a>0\) in \((r_-,r_+)\). The Wronskian argument above, applied with \(n=\ell\), shows that the radial eigenspace is one-dimensional. Since all other angular sectors have trivial kernel, and evenness in \(\theta\) excludes the sine mode, the kernel in the stated symmetry space is exactly \(\operatorname{span}\{q_a(r)\cos(\ell\theta)\}\).
\end{proof}

\subsection{The compact equation and its linearization}
We put the moving reference annuli inside one fixed domain and
linearize the compact equation used in the degree argument. Its
source increments vanish on the reference zero region, so the
quotient argument is the same as in the three-dimensional construction, now
with moving annuli and no curvature defect.

Fix $J=[a_0,a_1]$ and $\ell$ as above, and choose
\[
 D=\{d_-<|x|<d_+\},\qquad
 0<d_-<a_0-3<a_1+3<d_+.
\]
Let $\mathcal{X}\subset C_0(\overline D)$ be the closed real subspace
\[
 v(r,-\theta)=v(r,\theta),\qquad
 v(r,\theta+2\pi/\ell)=v(r,\theta),\qquad \|v\|_{\mathcal{X}}=\|v\|_{L^\infty(D)}.
\]
By~\eqref{p:eq:endpoint-growth}, the splitting
\eqref{common:eq-monotone-splitting} applies to $N=F_a$ on
$[0,\infty)$ with a fixed cutoff equal to one for $t\le t_0$ and
zero for $t\ge2t_0$. Extend $g_a$ oddly and set $k_a=0$ for negative
arguments. Then $g_a$ is continuous, strongly monotone and of at most
linear growth, while
\[
 k_a\in C^\infty(\overline D\times\RR),\qquad k_a\ge0, \quad \textup{and} \quad
 \qquad k_a(r,t)=0\quad \textup{for } t\le t_0.
\]
Note that these properties and smooth parameter dependence hold uniformly for
$a\in J$.

Proposition~\ref{common:prop-monotone-inverse}, with $\mu=1$ and zero
boundary data, gives $\mathcal T_aF=v$ solving
\[
 -\Delta v+g_a(|x|,v)=F\quad\hbox{in }D,\qquad v\in H_0^1(D).
\]
In particular, it follows that
\begin{equation}\label{p:eq:T-bound}
 \|\mathcal T_aF\|_{L^\infty(D)}\le\|F\|_{L^\infty(D)},\qquad
 \|\mathcal T_aF-\mathcal T_aG\|_{L^\infty(D)}\le\|F-G\|_{L^\infty(D)},
\end{equation}
and $\mathcal T_a[k_a(r,U_a)]=U_a$.
For the linear inverse, we set
\[
 V_a=\partial_tg_a(r,U_a)\ge1,\qquad
 \mathcal E_a=\mathcal E(\Omega_a,V_a),\qquad
 \mathfrak a_a(v,w)=\int_{\Omega_a}(\nabla v\cdot\nabla w+V_avw)\dd x.
\]
Note that, in the endpoint collars, $V_a=P_a\approx\log^2(1/d_a)/d_a^2$, where
$d_a=\dist(\cdot,\partial\Omega_a)$ refers to the reference annulus.
After shifting by $a$, these potentials have a common endpoint weight
and vary smoothly on compact interior sets. Thus,
Proposition~\ref{common:prop-singular-resolvent} gives
$\mathcal R_a=(-\Delta+V_a)^{-1}:L^2(D)\to\mathcal E_a$,
with outputs extended by zero to $D$, and all its uniform bounds
and continuity properties.

Next, we set
\[
 \mathcal K_a(v)=\mathcal T_a[k_a(r,U_a+v)]-U_a,\qquad
 b_a=\partial_tk_a(r,U_a)\ge0,\qquad
 \mathcal B_av=\mathcal R_a(b_av),
\]
and stress that radial coefficients and uniqueness preserve the symmetries of $\mathcal{X}$. Also, Proposition \ref{common:prop-monotone-inverse} and the smooth parameter dependence of \(k_a\) and \(U_a\) imply that \((a,v)\mapsto\mathcal K_a(v)\) is continuous and maps bounded subsets of \(J\times \mathcal{X}\) into relatively compact subsets of \(\mathcal{X}\). Moreover, \(\mathcal K_a(0)=0\).
On the other hand, the smooth coefficients $b_a$ are supported at uniform positive distance inside $\Omega_a$. Their uniform continuity in $a$ and
\eqref{common:eq-R-continuity} make $\mathcal B_a$ continuous in
operator norm. This linear family is also collectively compact.

\begin{proposition}\label{p:prop:composite-linearization}
Uniformly for $a\in J$,
\begin{equation}\label{p:eq:K-expansion}
 \mathcal K_a(v)=\mathcal B_av+\mathcal N_a(v),\qquad
 \sup_{a\in J}\frac{\|\mathcal N_a(v)\|_{\mathcal{X}}}{\|v\|_{\mathcal{X}}}
 \longrightarrow0,\quad \|v\|_{\mathcal{X}}\to0.
\end{equation}
\end{proposition}
\begin{proof}
Take $a_j\to a\in J$ and nonzero $v_j\to0$ in $\mathcal{X}$. Write
\[
 \varepsilon_j=\|v_j\|_{\mathcal{X}},\qquad
 \varphi_j=\mathcal T_{a_j}[k_{a_j}(r,U_{a_j}+v_j)],\qquad
 z_j=\frac{\varphi_j-U_{a_j}}{\varepsilon_j},
\]
and set
\[
 \eta_j=\frac{k_{a_j}(r,U_{a_j}+v_j)-k_{a_j}(r,U_{a_j})}
                  {\varepsilon_j}.
\]
The Lipschitz bound~\eqref{p:eq:T-bound} bounds $z_j$ uniformly, and
smoothness of $k_a$ gives
\[
 \eta_j=b_{a_j}\frac{v_j}{\varepsilon_j}+O(\varepsilon_j)
 \quad\hbox{in }L^\infty(D).
\]
Also, we know that $\varphi_j\ge0$, since $k_a\ge0$ for all real arguments. Subtracting
the inverse equations gives
$-\Delta z_j+c_jz_j=\eta_j$, where $c_j\ge1$ is the secant slope of
$g_{a_j}$ between $\varphi_j$ and $U_{a_j}$. At coincident positive values
use the derivative, and at a common zero set $c_j=1$. The product
$c_jz_j$ is then bounded for each fixed $j$.

On compact subsets of $\Omega_a$, the slopes converge uniformly to
$V_a$. In a two-sided reference collar of width $2\delta$,
nonnegativity and $\|\varphi_j-U_{a_j}\|_{\mathcal{X}} \le C\varepsilon_j$ give
$0\le \varphi_j,U_{a_j}\le2H(2\delta)$ for large $j$. Hence,
\eqref{p:eq:endpoint-growth} and~\eqref{common:eq-inverse-scale} imply
\[
 c_j\ge c\,\frac{\log^2(1/\delta)}{\delta^2}
 \quad\hbox{where }z_j\ne0,
\]
for some uniform constant $c > 0$. Also, the secants ending at zero satisfy the same bound by integration of the
derivative, which is locally integrable by~\eqref{p:eq:nearlinear-collar}.
On $D\setminus\Omega_{a_j}$, both terms defining $\eta_j$ vanish
for large $j$, since $U_{a_j}=0$ and $k_a(r,t)=0$ for $t\le t_0$.
 
Proposition~\ref{common:prop-linearization} therefore applies with
zero exterior source and gives
$\|z_j-\mathcal R_{a_j}\eta_j\|_{C(\overline{D})}\to0$.
The source expansion and the uniform bound for $\mathcal R_a$ prove
the claim.

Compactness of $J$ gives uniformity in $a$.
\end{proof}

\subsection{The degree argument}
The two endpoint operators have respectively two and three negative
directions, so their local degrees are $+1$ and $-1$. If no solution
met a small sphere as $a$ varied, homotopy invariance would make
these degrees equal. We use this contradiction to obtain a sequence,
and then normalize it to identify the angular deformation.

\begin{proposition}\label{p:prop:bifurcation}
There are $a_j\to a_*\in(a_0,a_1)$ and nonzero $v_j\in \mathcal{X}$ such that
\begin{equation*}
 v_j=\mathcal K_{a_j}(v_j),\qquad\norm{v_j}_{\mathcal{X}}\longrightarrow0.
\end{equation*}
After extracting a subsequence, and choosing a positive radial
$\ell$-mode kernel function $q_*$ normalized so that $\norm{q_*}_{\mathcal{X}}=1$, it follows that
\begin{equation}\label{p:eq:tangent}
 \frac{v_j}{\norm{v_j}_{\mathcal{X}}}
 \longrightarrow \sigma_* q_*(r)\cos(\ell\theta)
 \quad\hbox{uniformly on }\overline D,
 \qquad \sigma_*\in\{-1,1\},
\end{equation}
where $q_*>0$ on $(a_*-3,a_*+3)$ and is zero outside it.
\end{proposition}
\begin{proof}
\step{From the energy index to degree.}
On the even, \(\ell\)-fold symmetric energy space \(\mathcal E_a^\ell\), the operator \(\mathcal B_a=\mathcal R_a(b_a\cdot)\) is compact: this follows from the compact embedding \(\mathcal E_a^\ell\hookrightarrow L^2(\Omega_a)\), boundedness of \(b_a\), and boundedness of \(\mathcal R_a:L^2(\Omega_a)\to\mathcal E_a\). Moreover,
\[
\mathfrak a_a(\mathcal B_av,w)
=\int_{\Omega_a}b_avw\,dx
=\mathfrak a_a(v,\mathcal B_aw),
\qquad
\mathfrak a_a(\mathcal B_av,v)\ge0,
\]
so it is self-adjoint and nonnegative in this inner product. Since
\(V_a-b_a=\partial_tF_a(r,U_a)\), the linearized quadratic form satisfies
\[
\mathfrak q_a(v,v)
:=\int_{\Omega_a}
\left(|\nabla v|^2+\partial_tF_a(r,U_a)v^2\right)\,dx
=\mathfrak a_a((\Id-\mathcal B_a)v,v).
\]
Note as well that the bounded coefficient \(b_a\) does not change the form domain.

Next, we compare the realizations of \(\mathcal B_a\) on \(\mathcal{X}\) and \(\mathcal E_a^\ell\). If \(v\in \mathcal{X}\) satisfies \(\mathcal B_av=\mu v\), with \(\mu\ne0\), then, it follows that
\[
v=\mu^{-1}\mathcal R_a(b_av)\in\mathcal E_a^\ell.
\]
More generally, if
\[
(\mathcal B_a-\mu \Id)v_0=0,
\qquad
(\mathcal B_a-\mu \Id)v_j=v_{j-1},
\]
then all these vectors belong to \(\mathcal E_a^\ell\), by induction using
\[
v_j=\mu^{-1}(\mathcal B_av_j-v_{j-1}).
\]
Conversely, if \(v\in\mathcal E_a^\ell\) satisfies \(\mathcal B_av=\mu v\), with \(\mu\ne0\), then
\[
-\Delta v+V_av=\mu^{-1}b_av
\quad\text{in }\Omega_a.
\]
Interior elliptic regularity gives \(v\in H^2_{\mathrm{loc}}(\Omega_a)\subset C^0_{\mathrm{loc}}(\Omega_a)\). Since \(b_a\) has compact support inside \(\Omega_a\), we obtain \(b_av\in L^\infty(D)\), after extension by zero. Proposition \ref{common:prop-singular-resolvent} applied to \(v=\mu^{-1}\mathcal R_a(b_av)\) therefore gives \(v\in \mathcal{X}\).
Thus, the two realizations have the same nonzero eigenvalues and eigenspaces. Moreover, self-adjointness in the energy space excludes nontrivial generalized eigenvectors, and the preceding induction excludes them in \(\mathcal{X}\) as well. Hence, the algebraic multiplicities agree.

By the spectral theorem and the identity for \(\mathfrak q_a\), each eigenvector of \(\mathcal B_a\) with eigenvalue \(\mu\) contributes the coefficient \(1-\mu\) to the quadratic form, while the form is positive on \(\ker\mathcal B_a\). Consequently, its Morse index \(m(a)\) equals the number of eigenvalues of \(\mathcal B_a\) greater than \(1\), counted with algebraic multiplicity in \(\mathcal{X}\). Whenever \(\Id-\mathcal B_a\) is invertible, the Leray–Schauder degree formula therefore gives
\[
\deg(\Id-\mathcal B_a,B_\rho,0)=(-1)^{m(a)},
\]
where \(B_\rho=\{v\in \mathcal{X}:\|v\|_{\mathcal{X}}<\rho\}\).

By Proposition~\ref{p:prop:spectral-input}, $\Id-\mathcal B_{a_i}$ has a
trivial kernel for $i \in \{0,1\}$, and is therefore invertible on $\mathcal{X}$ by Fredholm alternative. The Morse indices two and three computed in Proposition \ref{p:prop:spectral-input}, then give
\begin{equation}\label{p:eq:endpoint-degrees}
 \deg(\Id-\mathcal B_{a_0},B_\rho,0)=+1,
 \qquad
 \deg(\Id-\mathcal B_{a_1},B_\rho,0)=-1, \qquad \textup{for every } \rho > 0.
\end{equation}

\step{Small nontrivial solutions.}
By invertibility and~\eqref{p:eq:K-expansion}, there exists $\rho_0 > 0$ such that
$$
\|\mathcal N_{a_i}(v)\|_{\mathcal{X}} <\tfrac12\|(\Id-\mathcal B_{a_i})v\|_{\mathcal{X}}, \qquad \textup{for }v \in B_{\rho_0} \setminus\{0\} \textup{ and }i \in \{0,1\}. 
$$
Thus, the straight homotopies between the linear and nonlinear endpoint
maps, i.e.
\[
H_t^i(v)= v - \mathcal{B}_{a_i} v -t\mathcal N_{a_i}(v),
\qquad 0\le t\le1.
\]
have no zeros on $\partial B_\rho$ for every sufficiently small $\rho > 0$, and their local degrees agree
with~\eqref{p:eq:endpoint-degrees}. In particular, it follows that, for every sufficiently small \(\rho>0\),
\[
\deg(\Id-\mathcal K_{a_0},B_\rho,0)=1,
\qquad
\deg(\Id-\mathcal K_{a_1},B_\rho,0)=-1.
\]

This implies that, for each sufficiently small $\rho>0$, there must be $a\in J$ and $\norm{v}_{\mathcal{X}} =\rho$ with $v=\mathcal K_a(v)$. Otherwise, the parameter homotopy would have no boundary zero and would force the two endpoint degrees to agree. We then choose $\rho_j\downarrow0$ and corresponding solutions $a_j,v_j$. Compactness of $J$ gives convergence up to a subsequence. 

\step{The angular direction.}
We divide the equation by $\rho_j$ and, taking into account \eqref{p:eq:K-expansion}, we get that
\begin{equation*}
 w_j=\mathcal B_{a_j}w_j+o_{\mathcal{X}}(1),\qquad
 w_j=v_j/\rho_j,\qquad\norm{w_j}_{\mathcal{X}}=1.
\end{equation*}
Norm continuity and compactness of the linear family give a uniformly
convergent subsequence, with limit $w_*$ of norm one satisfying
$w_*=\mathcal B_{a_*}w_*$. The endpoint operators have no kernel, so
$a_*$ is strictly inside $J$. The identity of energy and Banach kernels,
together with Proposition~\ref{p:prop:spectral-input}, now identifies
$w_*$ with the simple $\ell$-mode kernel, giving~\eqref{p:eq:tangent}.
\end{proof}

\subsection{The smooth Euler flows and their local geometry}
The localization argument of Subsection~\ref{s:sec:localization}
applies with one zero plateau and no drift. We first recover smooth
Euler flows, then use their nonzero angular coefficient to exclude
every possible center of local radiality.

\begin{proof}[Proof of Theorem~\ref{thm:planar-precise}]
Set $\psi_j=U_{a_j}+v_j$ and $\psi_*:=U_{a_*}$.
The fixed-point equation gives nonnegative, uniformly bounded solutions
to 
$$
\Delta \psi_j = F_{a_j}(|x|,\psi_j) \quad \textup{in } D, \qquad \psi_j = 0 \quad \textup{on } \partial D.
$$

\step{Localization and convergence.}
Let $C_D$ be the union of the fixed boundary collars, disjoint from every $\overline\Omega_{a_j}$,
we have $\|\psi_j\|_{L^\infty(C_D)}\le\|v_j\|_{L^\infty(C_D)}\to0$. Thus $k_{a_j}=0$ there
for large $j$, and then
\begin{equation*}
 -\Delta\psi_j+g_{a_j}(r,\psi_j)=0.
\end{equation*}
The endpoint formulas and $g_a=F_a$ at small positive values imply
\begin{equation*}
 g_a(r,H(t))\ge H''(t)-CH'(t)\ge\tfrac12H''(t)
\end{equation*}
for small $t$, uniformly in $a\in J$ and $r\in[d_-,d_+]$. Since $g_a$ is
increasing and vanishes at zero, Lemma~\ref{common:lem-exact-collar} applies
at both boundary circles. Hence, for a fixed $\delta_D>0$ and all large $j$,
\begin{equation*}
 \psi_j(x)=0\qquad\textup{if }\dist(x,\partial D)<\delta_D.
\end{equation*}

The zero extensions solve the same equation on a slightly larger
fixed annulus. The inverse bounds give uniform $W^{2,p}(D)$, and hence
$C^{1,\alpha}(\overline{D})$, estimates. With~\eqref{p:eq:F-nearlinear},
Theorem~\ref{common:lem-flat-regularity} now gives smoothness through
every zero, flatness there, and uniform bounds of every derivative
order, as in Subsection~\ref{s:sec:localization}. The supports lie
in a common compact annulus. Since $U_a$ depends smoothly on $a$,
we have $\psi_j\to U_{a_*}$ uniformly.  The uniform $C^{q+1}$ bounds
and compactness upgrade this to the desired $C^q$ convergence.

\step{The Euler equations and pressure.}
For large $j$, Lemma~\ref{p:lem:autonomy} applies. On $\{\psi_j>0\}$,
it gives $\nabla^\perp\psi_j\cdot\nabla\Delta\psi_j=0$.
Flatness at the zero set extends this identity to the whole plane. Thus, \(u_j=\nabla^\perp\psi_j \in C_c^{\infty}(\RR^2;\RR^2)\) is a stationary planar Euler flow for some smooth pressure \(p_j\). Since the pressure gradient has compact support, normalizing
its constant value in the unbounded exterior to zero gives
$p_j\in C_c^\infty(\RR^2)$.

\step{Local centers and exact symmetry.}
First, we choose a closed annular band $B$ inside the
increasing portion of $\psi_*$ so that
\begin{equation}\label{p:eq:band-gradient}
 \partial_r\psi_*\ge c_B>0\quad\textup{on }B.
\end{equation}
By $C^1$ convergence, it then follows that $\partial_r\psi_j\ge c_B/2$ on $B$ for sufficiently large $j$. We now fix a radius $r_0$ such that $\{r = r_0\} \subset B$. The convergence in \eqref{p:eq:tangent} then gives
\begin{equation} \label{eq:angular-fourier-coefficient}
 \frac1{\pi\|v_j\|_{\mathcal{X}}}\int_0^{2\pi}
       \psi_j(r_0,\theta)\cos(\ell\theta)\dd\theta
 \longrightarrow\sigma_*q_*(r_0)\ne0.
\end{equation}

Next, suppose by contradiction that \(\psi_j\) is locally radial. Since \(B\) is connected and \(\nabla\psi_j\ne0\) there, it lies in one connected component of \(\{\nabla\psi_j\ne0\}\), on which \(\psi_j\) is radial about a fixed center \(p\). This center is unique. Indeed,  radiality about two distinct centers would force the nonzero gradient to be parallel to both corresponding radial directions, which is impossible on an open set. Moreover, rotation through \(2\pi/\ell\) preserves \(B\) and \(\psi_j\), so \(\psi_j\) is radial on \(B\) about both \(p\) and its rotated image. Uniqueness of the center implies that these points coincide, and hence \(p=0\), contradicting \eqref{eq:angular-fourier-coefficient} for $j$ sufficiently large. Therefore \(\psi_j\) is not locally radial for $j$ sufficiently large. 

To identify its discrete rotational symmetries, we set
\[
c_j:=\int_0^{2\pi}\psi_j(r_0,\theta)e^{-i\ell\theta}\,d\theta.
\]
Using again \eqref{eq:angular-fourier-coefficient} we get that \(c_j\ne0\) for $j$ sufficiently large. If \(\psi_j\) is invariant under rotation through \(\alpha\), a change of variables yields
\[
c_j=e^{i\ell\alpha}c_j,
\]
and hence, as $c_j \neq 0$, \(\alpha\in(2\pi/\ell)\mathbb Z\). Since the symmetries defining \(\mathcal{X}\) already give invariance under rotation through \(2\pi/\ell\) and reflection across the horizontal axis, \(\psi_j\) has exactly the asserted symmetries.
\end{proof}

\begin{remark}\label{p:rem:rotating}
Applying the localization of rigid rotation from~\cite{EFRrot} to the
planar flows above also yields $C^\infty$ uniformly rotating Euler
flows with compactly supported velocity.
\end{remark}

\section*{Acknowledgements}
The authors warmly thank David Ruiz for many interesting discussions on the topic of this paper. In particular, he found a very useful simplification for the three-dimensional ansatz.

This work was supported by the European Research Council (ERC) under
the European Union's Horizon 2020 research and innovation programme
through grant agreement 862342 (A.E.). A.E. also acknowledges support
from the Spanish State Research Agency through grants
PID2022-136795NB-I00 and PID2025-168410NB-I00, and from the
ICMAT--Severo Ochoa grant CEX2023-001347-S. A.J.F. acknowledges support by the grants PID2023-149451NA-I00 of MCIN/AEI/\allowbreak 10.13039/\allowbreak 501100011033/FEDER,UE; and RYC2024-049142-I, MICINN (Spain).

\section*{Declaration on the use of AI tools}
The authors used ChatGPT to assist with mathematical arguments
and calculations, the identification of errors and imprecisions, and
the organization, drafting and revision of the manuscript. The authors
take full responsibility for the mathematical results, arguments and
final content of this paper.

\end{document}